\documentclass[leqno,envcountsame,envcountsect]{svproc}

\usepackage{amsmath}
\usepackage{amsfonts}
\usepackage{amssymb}
\usepackage{euscript}
\usepackage{upref}
\usepackage{mathrsfs}
\usepackage[all]{xy}
\usepackage{comment}

\usepackage{enumitem}
\setitemize{label={$\bullet$}}

\usepackage{url}

\spnewtheorem{observation}[theorem]{Observation}{\bfseries}{\itshape}
\spnewtheorem{nothing}[theorem]{}{\bfseries}{\itshape}
\spnewtheorem{nothing*}[theorem]{}{\bfseries}{\rmfamily}
\spnewtheorem{notation}[theorem]{Notation}{\bfseries}{\rmfamily}
\spnewtheorem{caution}[theorem]{Caution}{\bfseries}{\rmfamily}

\spnewtheorem{Definition}[theorem]{Definition}{\bfseries}{\rmfamily}
\spnewtheorem{Example}[theorem]{Example}{\bfseries}{\rmfamily}
\spnewtheorem*{Example*}{Example}{\bfseries}{\rmfamily}
\spnewtheorem{Remark}[theorem]{Remark}{\bfseries}{\rmfamily}
\spnewtheorem{Question}{Question}{\bfseries}{\rmfamily}

\DeclareMathOperator{\Aut}{Aut}

\DeclareMathOperator{\image}{im}
\DeclareMathOperator{\trdeg}{trdeg}

\DeclareMathOperator{\Char}{char}
\DeclareMathOperator{\Span}{Span}
\DeclareMathOperator{\id}{id}
\DeclareMathOperator{\Der}{Der}

\DeclareMathOperator{\Nil}{Nil}
\DeclareMathOperator{\UNil}{UNil}
\DeclareMathOperator{\lnd}{LND}
\DeclareMathOperator{\locnil}{LN}

\DeclareMathOperator{\ad}{ad}

\newcommand{\powerset}{	\raisebox{0.7\depth}{\large $\wp$} }
\newcommand{\smallpowerset}{	\raisebox{0.7\depth}{\small $\wp$} }
\newcommand{\powerfin}{	\powerset_{\!\mbox{\rm\tiny fin}}}

\newcommand{\setspec}[2]{\big\{\,#1\, \mid \,#2\, \big\}}

\newcommand{\Ascr}{\mathscr{A}}

\newcommand{\II}{\ensuremath{\mathbb{I}}}
\newcommand{\LL}{\text{\rm L}}

\newcommand{\Nat}{\mathbb{N}}

\newcommand{\bk}{\mathbf{k}}

\newcommand{\Leul}{\EuScript{L}}

\newcommand{\Veul}{\EuScript{V}}

\newcommand{\isom}{\cong}
\renewcommand{\epsilon}{\varepsilon}
\renewcommand{\phi}{\varphi}
\renewcommand{\emptyset}{\varnothing}

\newenvironment{enumerata}%
{\begin{enumerate}[{label={\normalfont(\alph*)}}]}
{\end{enumerate}}

\begin{document}
\mainmatter         
\title{Locally nilpotent sets of derivations}
\titlerunning{Locally nilpotent sets of derivations}
\author{Daniel Daigle\thanks{Research supported by grant RGPIN/2015-04539 from NSERC Canada.}} 
\authorrunning{Daniel Daigle} 
\tocauthor{Daniel Daigle}
\institute{Department of Mathematics and Statistics \\ University of Ottawa, Ottawa (ON), Canada K1N 6N5 \\
\email{ddaigle@uottawa.ca}}

\maketitle           

{\renewcommand{\thefootnote}{}
}

\begin{abstract}
Let $B$ be an algebra over a field $\bk$.
We define what it means for a subset of $\Der_\bk(B)$ to be a {\it locally nilpotent set}.
We prove some basic results about that notion and explore the following questions.
Let $L$ be a Lie subalgebra of $\Der_\bk(B)$;
if $L \subseteq \lnd(B)$ then does it follow that $L$ is a locally nilpotent set? Does it follow that $L$ is a nilpotent Lie algebra?
\end{abstract}


In this article, an {\it algebra} over a field $\bk$ is a pair $(B,\cdot)$ where $B$ is a $\bk$-vector space and ``$\cdot$'' is
an arbitrary $\bk$-bilinear map $B \times B \to B$, $(x,y) \mapsto x \cdot y$.
Let $B$ be an algebra over a field $\bk$ and let $\Der_\bk(B)$ be the set of all $\bk$-derivations $D : B \to B$.
A derivation $D \in \Der_\bk(B)$ is said to be \textit{locally nilpotent} if for each $x \in B$ there exists an $n>0$ such that $D^n(x)=0$.
We write $\lnd(B)$ for the set of locally nilpotent derivations of $B$.
We say that a subset $\Delta$ of $\Der_\bk(B)$ is \textit{locally nilpotent} if for each $x \in B$ the following holds:
\begin{quote}
for every infinite sequence $(D_1, D_2, \dots)$ of elements of $\Delta$, there exists $n$\\
such that $(D_n \circ \cdots \circ D_1)(x)=0$.
\end{quote}
We say that $\Delta$ is \textit{uniformly locally nilpotent} if for each $x \in B$ there exists $n$ such that
$(D_n \circ \cdots \circ D_1)(x)=0$ for all $(D_1,\dots,D_n) \in \Delta^n$.
Then it is clear that the implications 
$$
\text{$\Delta$ is uniformly locally nilpotent} \, \Rightarrow \,
\text{$\Delta$ is locally nilpotent} \, \Rightarrow \,
\Delta \subseteq \lnd(B)
$$
are true, and it is easy to see that both converses are false.
Since $\Delta \subseteq \lnd(B)$ does not imply that $\Delta$ is locally nilpotent, it is natural to ask whether the stronger condition
 $\Span_\bk(\Delta) \subseteq \lnd(B)$ would imply that $\Delta$ is locally nilpotent;
it turns out that the answer is negative:

\begin{Example*} 
Let $\bk$ be a field and $B = \bk[x,y,z]$ (polynomial ring in $3$ variables).
Let $\Delta = \{D,E\}$ where $D = x \frac{\partial}{\partial y} + y \frac{\partial}{\partial z}$ and
$E = y \frac{\partial}{\partial x} - z \frac{\partial}{\partial y}$.
Then $\Span_\bk(\Delta) \subseteq \lnd(B)$. However, we have $D(E(x))=x$, so $\Delta$ is not a locally nilpotent set.
\end{Example*}

In the above Example
the derivation $[D,E]= D \circ E - E \circ D$ is not locally nilpotent (it sends $x$ to itself),
so the Lie algebra generated by $\Delta$ is not included in $\lnd(B)$.
So it makes sense to ask: if we make the stronger assumption that the Lie algebra generated by $\Delta$ is included in $\lnd(B)$,
then does it follow that $\Delta$ is locally nilpotent?
This is equivalent to asking:
\begin{Question}
Let $L$ be a Lie subalgebra of $\Der_\bk(B)$.
If $L \subseteq \lnd(B)$, does it follow that $L$ is a locally nilpotent subset of  $\Der_\bk(B)$?
\end{Question}
This question leads naturally to another question: can we characterize the Lie algebras that can be embedded in $\lnd(B)$?
This question may be too ambitious, so we shall restrict ourselves to the following version of it:

\begin{Question}
Let $L$ be a Lie subalgebra of $\Der_\bk(B)$.
If $L$ satisfies one of:\\
$\bullet$  $L \subseteq \lnd(B)$,\\
$\bullet$  $L$ is a locally nilpotent (or uniformly locally nilpotent) subset of $\Der_\bk(B)$,\\
$\bullet$ $L$ is a Lie-locally nilpotent (or uniformly so) subset of $\Der_\bk(B)$,\\
then does it follow that $L$ is a nilpotent Lie algebra, or that $L$ satisfies some other nilpotency condition that makes sense
for abstract Lie algebras?
(The notion of a Lie-locally nilpotent subset of $\Der_\bk(B)$ is defined in Sec.\ \ref{SEC:LieLocallynilpotentsetsoflinearmaps}.)
\end{Question}

The aim of this article is to develop the basic theory of locally nilpotent sets and to explore what can be said about Questions 1 and 2
(where in fact Question 2 contains several questions).
The basic theory is developed in Section \ref{SEC:Locallynilpotentactions}; the first steps are done in the more general context of linear endomorphisms of
vector spaces, and the later part of the Section is devoted to derivations.
Section \ref{SEC:LieLocallynilpotentsetsoflinearmaps} deals with a variant of the notion of locally nilpotent set
that is relevant in the context of Lie algebras of derivations.
Section \ref{SEC:Someapplicationstoalgebras} studies five notions of nilpotency for Lie or associative algebras,
needed in order to address Question 2.
Three of them are classical (an algebra can be \textit{nilpotent}, \textit{nil}, or \textit{locally nilpotent}).
The other two are those of a \textit{sequentially nilpotent} algebra and of a \textit{locally nil} algebra;
as far as we know these two have not been studied previously. 
Sections \ref{SEC:Localnilpotenceanduniformlocalnilpotenceforderivations} and \ref{SEC:thecaseofderivationfinitealgebras} are mainly devoted
to Questions 1 and 2.
Section \ref{SEC:Localnilpotenceanduniformlocalnilpotenceforderivations} shows that if $B$ is the polynomial ring in $| \Nat |$ variables
over a field $\bk$ then the following hold:
\begin{enumerate}[{label={\normalfont(\arabic*)}}]

\item (Ex.\ \ref{298ehr0c9vvb349t9pwqnfvw9e8}) There exists a Lie subalgebra $L$ of $\Der_\bk(B)$ such that
\begin{enumerate}

\item every finitely generated Lie subalgebra of $L$ is a  locally nilpotent subset of $\Der_\bk(B)$ (so $L \subseteq \lnd(B)$);

\item $L$ is not a locally nilpotent subset of $\Der_\bk(B)$.

\end{enumerate}
So Question 1 has a negative answer.

\item (Ex.\ \ref{cp0b3497m5yokejdrjmk6yf}) For each integer $m\ge2$ there exists an $m$-generated Lie subalgebra $L$ of $\Der_\bk(B)$ satisfying:
\begin{enumerate}

\item every $(m-1)$-generated Lie subalgebra of $L$ is a locally nilpotent subset of $\Der_\bk(B)$  (so $L \subseteq \lnd(B)$);

\item $L$ is not a locally nilpotent subset of $\Der_\bk(B)$.

\end{enumerate}
So Question 1 has a negative answer even when $L$ is a finitely generated Lie algebra.

\item (Ex.\ \ref{2ndPart-Zf24obo87vWLkhkjDTfo499476df}) There exists a Lie subalgebra $L$ of $\Der_\bk(B)$ such that
\begin{enumerate}

\item $L$ is a uniformly locally nilpotent subset of $\Der_\bk(B)$;

\item $L$ is the free Lie algebra on a countably infinite set (so $L$ is as non-nilpotent as a Lie algebra can be).

\end{enumerate}
So all questions that are part of Question 2 have negative answers.

\item (Cor.\ \ref{p0c9vn349Cf0cpe0a})  There exists an infinite subset $\Delta$ of $\Der_\bk(B)$ satisfying:
\begin{enumerate}

\item $\Delta$ is not a locally nilpotent subset of $\Der_\bk(B)$;

\item every finite subset of $\Delta$ is a locally nilpotent subset of $\Der_\bk(B)$.

\end{enumerate}

\item (Cor.\ \ref{p0c9vn349Cf0cpe0a}) For each integer $m\ge2$, there exists an $m$-subset $\Delta$ of $\Der_\bk(B)$ satisfying:
\begin{enumerate}

\item $\Delta$ is not a locally nilpotent subset of $\Der_\bk(B)$;

\item every proper subset of $\Delta$ is a locally nilpotent subset of $\Der_\bk(B)$.

\end{enumerate}
\end{enumerate}
Facts (4) and (5) show just how capriciously the notion of locally nilpotent set can behave.
Note that (2) and (5) are based on Golod's famous example 
of an associative algebra that is nil and finitely generated but not nilpotent.

The above facts suggest that if we do not assume that $B$ satisfies some finiteness condition
then we cannot expect our questions to have affirmative answers.
Section \ref{SEC:thecaseofderivationfinitealgebras} re-examines Question 2 under
the additional assumption that $B$ is ``derivation-finite''
(this is more general than $B$ being finitely generated, see Def.\ \ref{c9fweCdefderfinitedc9v9ueibqovp}).
Cor.\ \ref{cijvnn394npebg67872u} and Prop.\ \ref{p90cb2473dq739e3r627} give some positive answers to Question 2, but Question 1 is left open.

\medskip
In future work, we hope to apply these ideas to the following questions.
Let $B$ be a commutative algebra over a field $\bk$ of characteristic zero, let $\Delta$ be a subset of $\lnd(B)$ and consider
the subset $Y = \setspec{ \exp(D) }{ D \in \Delta }$ of $\Aut_\bk(B)$.
{\it How are the properties of $\Delta$ related to those of $Y$? When is $Y$ a group? When is $Y$ included in an algebraic group?}

\section{Preliminaries}

\begin{nothing*}  \label {09cn3049dEw}
\textbf{Directed trees.}
Consider a pair $(S,E)$ where $S$ is a set and $E$ is a subset of $S \times S$ such that every element $(u,v)$ of $E$ satisfies $u \neq v$.
Then $(S,E)$ can be viewed as a directed graph where $S$ is the vertex-set and $E$ is the edge-set
(if $(u,v) \in E$ then $(u,v)$ is a directed edge from $u$ to $v$).
We allow $S$ and $E$ to be infinite sets.
By a \textit{path} in the graph $(S,E)$,
we mean a sequence $\gamma = (x_0, x_1, x_2, \dots )$ (finite or infinite) of elements of $S$
such that $(x_i,x_{i+1}) \in E$ for all $i \ge0$.
If $\gamma = (x_0, x_1, \dots, x_n)$ is a finite path, we say that $\gamma$ is a path {\it from $x_0$ to $x_n$.}
If $x \in S$ then we regard $(x)$ as a path from $x$ to $x$.

Now suppose that $(S,E)$ is a pair of the above type. Then it is not hard to see that there exists at most one vertex $v_0 \in S$ satisfying:
for every $v \in S$ there exists exactly one path from $v_0$ to $v$.
If such a vertex $v_0$ exists then we say that \textit{$(S,E)$ is a directed tree with root $v_0$}.
We shall use the following fact several times; its proof is left to the reader.
\end{nothing*}  

\begin{theorem}  \label {0veohtg09ernd0}
Let $(S,E)$ be a directed tree with root $v_0$, and suppose:
\begin{itemize}

\item for each $v \in S$, there exist finitely many $v' \in S$ such that $(v,v') \in E$;

\item there exists no infinite path $(x_0,x_1,x_2,\dots)$ in $(S,E)$ satisfying $x_0 = v_0$.

\end{itemize}
Then $S$ is a finite set.
\end{theorem}

\begin{nothing*}  \label {0bg93mm5k6gnndrtuiecj8}
\textbf{Algebras.}
Let $\bk$ be a field.

\smallskip

\noindent (a)\ \  A {\it $\bk$-algebra} is a pair $(A,\cdot)$ where $A$ is a $\bk$-vector space and ``$\cdot$'' is
an arbitrary $\bk$-bilinear map $A \times A \to A$, $(x,y) \mapsto x \cdot y$.
We call $\cdot$ the multiplication of $A$.

Let $A$ be a $\bk$-algebra.

A {\it subalgebra} of $A$ is a $\bk$-subspace $A'$ of $A$ closed under the multiplication of $A$.
If $\Delta$ is a subset of $A$ then the {\it subalgebra of $A$ generated by $\Delta$} is the intersection of all subalgebras $A'$ of $A$
satisfying $\Delta \subseteq A'$.

A {\it $\bk$-derivation} of $A$ is a $\bk$-linear map $D : A \to A$ satisfying
$D(x \cdot y) = D(x) \cdot y + x \cdot D(y)$ for all $x,y \in A$. 

A {\it homomorphism} $\phi : A \to B$  of $\bk$-algebras is a $\bk$-linear map satisfying $\phi(x \cdot y) = \phi(x) \cdot \phi(y)$
for all $x,y \in A$.

\smallskip

\noindent (b)\ \ An \textit{associative algebra} over $\bk$ is a $\bk$-algebra whose multiplication is associative.

\smallskip

\noindent (c)\ \ A {\it Lie algebra} over $\bk$ is a $\bk$-algebra $(A,\cdot)$ 
satisfying $x \cdot x = 0$ and the Jacobi identity $x \cdot (y \cdot z) + y \cdot (z \cdot x) + z \cdot (x \cdot y) = 0$
for all $x,y,z \in A$.

\smallskip

For associative algebras and Lie algebras,
the notions of {\it subalgebra}, {\it subalgebra generated by a set}, {\it homomorphism} and {\it $\bk$-derivation}
are those defined in part (a) for general algebras.
\end{nothing*}

\begin{notation}  \label {p0vQ398vf83ms43rt7uxcy9w47F}
\noindent (a)\ \ If $(A,\cdot)$ is an associative algebra and if we define $a_1 * a_2 = a_1 \cdot a_2 - a_2 \cdot a_1$  for all $a_1,a_2 \in A$
then $(A,*)$ is a Lie algebra that we denote $A_\LL$. 

\smallskip

\noindent (b)\ \ If $V$ is a vector space over $\bk$ then $\Leul_\bk(V)$ denotes the set of all $\bk$-linear maps $V \to V$.
We always regard $\Leul_\bk(V)$ as an associative algebra, the multiplication being the composition of maps.
We write $\Leul_\bk(V)_\LL$ for the corresponding Lie algebra, as explained in part (a).

\smallskip

\noindent (c)\ \ Let $A$ be an algebra over a field $\bk$.
The set of all $\bk$-derivations of $A$ is denoted $\Der_\bk(A)$, and is a $\bk$-subspace of $\Leul_\bk(A)$.
If $D,E \in \Der_\bk(A)$ then $D \circ E - E \circ D \in \Der_\bk(A)$, so $\Der_\bk(A)$ is a subalgebra of the Lie algebra $\Leul_\bk(A)_\LL$.
\end{notation}

\begin{notation}  \label {hfc026mx4528sxgs}
Let $A$ be an associative algebra over a field $\bk$.
If $\Delta$ is any subset of $A$ then $\bar\Delta$ denotes the subalgebra of $A$ generated by $\Delta$ 
and $\tilde\Delta$ denotes the subalgebra of $A_\LL$ generated by $\Delta$.
We have $\Delta \subseteq \tilde\Delta \subseteq \bar\Delta$.

Note that $\bar\Delta$ is the intersection of all $\bk$-subspaces $A'$ of $A$ that are closed under the multiplication of $A$
and satisfy $\Delta \subseteq A'$.  We also point out that $\bar\Delta = \Span_\bk( \Delta_\circ )$ where
$\Delta_\circ$ is the set of \textbf{nonempty} products of elements of $\Delta$.

The notation $\bar \Delta$ has the above meaning even when $A$ happens to have a unity $1$.
For instance, if $A$ is the commutative polynomial ring $\bk[x,y]$ and $\Delta = \{x,y\}$ then $\bar\Delta$ is the ideal generated by $x,y$.
\end{notation}

\begin{nothing*}
We adopt the  \textbf{right-associativity convention},  which stipulates that
any unparenthesized product $a_n \cdots a_1$ (where $n>2$ and $a_1, \dots, a_n$ are elements of some algebra)
is to be interpreted as meaning  $a_n \cdot ( a_{n-1} \cdots ( a_3 \cdot (a_2 \cdot a_1)) \dots)$.
For instance, $a \cdot b \cdot c \cdot d = a \cdot (b \cdot (c \cdot d))$.
This convention is in effect throughout the paper, in all types of algebras.
\end{nothing*}

\begin{nothing*}  \label {pa9enbr2o39efcp230wsd}
Let $(A,\cdot)$ be a Lie algebra.
It is customary to use the bracket notation for multiplication, i.e., to define $[x,y] = x \cdot y$ for all $x,y \in A$.
We shall use both notations. 
If $x_0,x_1,\dots,x_n \in A$, the element $[x_n, [x_{n-1}, \dots [x_2,[x_1,x_0]] \dots ]]$ of $A$ is simply denoted $[x_n, \dots, x_0]$.
This agrees with the right-associativity convention; for instance, going back and forth between the two notations,
$$
[x_3,x_2,x_1,x_0] = [ x_3, [ x_2, [ x_1, x_0 ]]] = x_3 \cdot (x_2 \cdot (x_1 \cdot x_0)) = x_3 \cdot x_2 \cdot x_1 \cdot x_0\;.
$$
In general we have $[x_n, \cdots, x_0] = x_n \cdots x_0$, where we allow the case $n=0$: $[x_0] = x_0$.

If $x \in A$ then the map $y \mapsto [x,y]$ is denoted $\ad(x) : A \to A$.
The following properties of $\ad$ are well known, and easily verified:\\
(i) $\ad : A \to \Leul_\bk(A)_\LL$ is a homomorphism of Lie algebras;\\
(ii) $\ad(A) \subseteq \Der_\bk(A)$.\\
Note that if $n>0$ and $x_0, \dots, x_n \in A$ then 
$$
[x_n, \dots, x_0] = \big( \ad(x_n) \circ \cdots \circ \ad(x_1) \big) (x_0)\;.
$$
\end{nothing*}

We leave it to the reader to check the following fact (which is trivial in the associative case):

\begin{lemma}  \label {0vh349rf920}
Let $(A,\cdot)$ be an associative algebra or a Lie algebra over a field~$\bk$.
\begin{enumerata}

\item Let $x_1, \dots, x_n \in A$ and let $x \in A$ be any parenthesization of the product $x_n \cdots x_1$.
Then $x$ is a finite sum $\sum_i r_i \, x_{i,n} \cdots x_{i,1}$ where, for each $i$,
$r_i$ belongs to the prime subring of $\bk$ and $(x_{i,1}, \dots, x_{i,n})$ is a permutation of $(x_1, \dots, x_n)$.

\item If $H$ is a generating set for $A$  then 
$$
A = \Span_\bk \setspec{ x_m \cdots x_0 }{ \text{$m\in\Nat$ and $(x_0, \dots, x_m) \in H^{m+1}$} } .
$$

\end{enumerata}
\end{lemma}

\section{Locally nilpotent sets of linear maps}
\label {SEC:Locallynilpotentactions}

Throughout this section, we consider a vector space $V$ over a field $\bk$
and we let $\Leul_\bk(V)$ denote the set of all $\bk$-linear maps $V \to V$.
An element $F$ of $\Leul_\bk(V)$ is said to be \textit{locally nilpotent} if for each $x \in V$ there exists $n>0$ such that $F^n(x)=0$.
We write $\locnil( \Leul_\bk V )$ for the set of all locally nilpotent elements of $\Leul_\bk(V)$.

If $\Delta$ is a set then $\Delta^\Nat$ denotes the set of all infinite sequences $(a_0,a_1,a_2, \dots)$ of elements of~$\Delta$.

\begin{Definition} \label {0f9v358f6c7we9ww5gr}
Given a subset $\Delta$ of $\Leul_\bk(V)$, we define the subsets $\Nil(\Delta)$ and $\UNil(\Delta)$ of $V$ as follows.
If $\Delta=\emptyset$, we set $\Nil(\Delta)=V=\UNil(\Delta)$.
If $\Delta \neq \emptyset$,
\begin{itemize}

\item $\Nil(\Delta)$ is the set of all $x \in V$ satisfying:
\begin{quote}
for every sequence $(F_0,F_1,\dots) \in \Delta^\Nat$ there exists $n \in \Nat$ such that\\
$(F_n \circ \cdots \circ F_0)(x) =0$;
\end{quote}

\item $\UNil(\Delta)$ is the set of all $x \in V$ satisfying:
\begin{quote}
there exists $n \in \Nat$ such that for every $(F_1,\dots,F_n) \in \Delta^n$ we have 
$(F_n \circ \cdots  \circ F_1)(x) =0$.
\end{quote}

\end{itemize}
Note that  $\UNil(\Delta) \subseteq \Nil(\Delta)$ are linear subspaces of $V$.
If $\Nil(\Delta) = V$, we say that
$\Delta$ is a \textit{locally nilpotent subset of $\Leul_\bk(V)$} (or simply, that $\Delta$ is \textit{locally nilpotent});
if $\UNil(\Delta) = V$, we say that $\Delta$ is \textit{uniformly locally nilpotent}.
\end{Definition}

If $\Delta$ is a locally nilpotent subset of $\Leul_\bk(V)$ then $\Delta \subseteq \locnil( \Leul_\bk V )$;
however, the converse is not true.

If $\Delta$ is a uniformly locally nilpotent subset of $\Leul_\bk(V)$ then $\Delta$ is locally nilpotent,
but the converse is not true (see Ex.\ \ref{928349rfer8d9}).
However, note the following:

\begin{lemma} \label {0cccno12q9wdnp}
Let $\Delta$ be a finite subset of $\Leul_\bk(V)$. Then $\Nil(\Delta) = \UNil(\Delta)$.
In particular, $\Delta$ is locally nilpotent if and only if it is uniformly locally nilpotent.
\end{lemma}

\begin{proof}
It suffices to show that  $\Nil(\Delta) \subseteq \UNil(\Delta)$.
Let  $x \in \Nil(\Delta) \setminus \{0\}$.
Consider the set $S$ of finite sequences $(F_1, \dots, F_n)$ of elements of $\Delta$
satisfying $(F_n \circ \cdots\circ  F_1)(x) \neq 0$ (where the empty sequence $()$ belongs to $S$).
Let $E \subseteq S \times S$ be the set of ordered pairs of the form
$\big( (F_1, \dots, F_n), (F_1, \dots, F_n, F_{n+1}) \big)$.
Then $(S,E)$ is a directed tree with root $()$ (see \ref{09cn3049dEw}).
For each vertex $v \in S$, the number of $v'\in S$ satisfying $(v,v') \in E$ is at most $|\Delta|$, which is finite.
The fact that $x \in \Nil(\Delta)$ implies that there is no infinite path $\big( (), (F_1), (F_1,F_2), \dots \big)$ in this tree.
It follows from Thm \ref{0veohtg09ernd0} that $S$ is a finite set, and this implies that $x \in \UNil(\Delta)$.
\qed\end{proof}

\begin{Definition} \label {02n3Ffjiw39dg7tgflfia}
Given a nonempty subset $\Delta$ of $\Leul_\bk(V)$ we define a map $\deg_\Delta : V \to \Nat \cup \{ -\infty, \infty\}$
by declaring that
$\deg_\Delta(0) = -\infty$,
 if $x \in V \setminus \UNil(\Delta)$ then $\deg_\Delta(x) = \infty$,
and if $x \in \UNil(\Delta) \setminus \{0\}$ then $\deg_\Delta(x)$ is equal to:
$$
\max\setspec{ n \in \Nat }{ \text{there exists $(F_1, \dots, F_n) \in \Delta^n$ satisfying $(F_n \circ \cdots \circ F_1)(x) \neq 0$} }
$$
(to be clear,
if $x \neq 0$ and $F(x)=0$ for all $F\in\Delta$ then we define $\deg_\Delta(x)=0$).
If $\Delta = \emptyset$ then we define 
$\deg_\Delta(0) = -\infty$ and $\deg_\Delta(x) = 0$ for all $x \in V \setminus \{0\}$.
\end{Definition}

\begin{lemma}  \label {kjcoo93edg98uwb}
Let $\Delta$ be a subset of $\Leul_\bk(V)$.  Then the map $\deg_\Delta : V \to \Nat \cup \{ \pm\infty \}$
has the following properties, where $x,y$ are arbitrary elements of $V$.
\begin{enumerata}

\item $\deg_\Delta(x) < 0$ if and only if $x=0$.

\item $\deg_\Delta(x) \le 0$ if and only if $F(x) = 0$ for all $F \in \Delta$.

\item $\deg_\Delta(x) < \infty$ if and only if $x \in \UNil(\Delta)$.

\item If $x  \in \UNil(\Delta) \setminus\{0\}$ and $F \in \Delta$ then $\deg_\Delta(F(x)) < \deg_\Delta(x)$.

\item $\deg_\Delta(x+y) \le \max( \deg_\Delta(x), \deg_\Delta(y) )$  

\end{enumerata}
\end{lemma}

Verification of Lemma \ref{kjcoo93edg98uwb} is left to the reader.

The notations $\bar\Delta$ and $\tilde\Delta$ are defined in \ref{hfc026mx4528sxgs}.

\begin{proposition}  \label {d0c1i2eb9ew03j}
Let $\Delta$ be a subset of $\Leul_\bk(V)$,
let $\bar\Delta$ be the associative subalgebra of $\Leul_\bk(V)$ generated by $\Delta$
and let $\tilde\Delta$ be the Lie subalgebra of $\Leul_\bk(V)_\LL$ generated by $\Delta$.
Then $\Delta \subseteq \tilde\Delta \subseteq \bar\Delta \subseteq \Leul_\bk(V)$ and the following hold.
\begin{enumerata}

\item $\Nil(\bar\Delta) = \Nil(\tilde\Delta) = \Nil(\Delta)$

\item $\deg_{\bar\Delta} = \deg_{\tilde\Delta} = \deg_{\Delta}$ and $\UNil(\bar\Delta) = \UNil(\tilde\Delta) = \UNil(\Delta)$

\item If $\Delta$ is locally nilpotent then so are $\bar\Delta$ and $\tilde\Delta$.
\item If $\Delta$ is uniformly locally nilpotent then so are $\bar\Delta$ and $\tilde\Delta$.

\end{enumerata}
\end{proposition}

\begin{proof}
We may assume that $\Delta \neq \emptyset$.
Let $\Delta_\circ$ be as in \ref{hfc026mx4528sxgs} and observe that 
$\Delta \subseteq \Delta_\circ \subseteq \bar\Delta$
and $\Delta \subseteq \tilde\Delta \subseteq \bar\Delta$.
These inclusions have the following trivial consequences:
\begin{gather}
\label {9812hsj2k4gs93u} \text{$\Nil(\bar\Delta) \subseteq \Nil(\Delta_\circ) \subseteq \Nil(\Delta)$
and $\Nil(\bar\Delta) \subseteq \Nil(\tilde\Delta) \subseteq \Nil(\Delta)$,} \\
\label {tw8254hdsdh2ireg545j} \text{$\UNil(\bar\Delta) \subseteq \UNil(\Delta_\circ) \subseteq \UNil(\Delta)$
and $\UNil(\bar\Delta) \subseteq \UNil(\tilde\Delta) \subseteq \UNil(\Delta)$,} \\
\label {93hfd54e15syu3rtdj} \begin{minipage}[t]{.9\textwidth}
for all $x \in V$ we have\\
$\deg_{\Delta}(x) \le \deg_{\Delta_\circ}(x) \le \deg_{\bar\Delta}(x)$
and $\deg_{\Delta}(x) \le \deg_{\tilde\Delta}(x) \le \deg_{\bar\Delta}(x)$.
\end{minipage}
\end{gather}

Let $x \in \Nil(\Delta)$.
If $(G_0,G_1,\dots) \in \Delta_\circ^\Nat$ then there exist $(F_0,F_1,\dots ) \in \Delta^\Nat$ and $1 \le n_0 < n_1 < \cdots$ in $\Nat$ such that
$G_0 = F_{n_0-1} \circ \cdots \circ F_0$, 
$G_1 = F_{n_1-1} \circ \cdots \circ F_{n_0}$, 
$G_2 = F_{n_2-1} \circ \cdots \circ F_{n_1}$, and so on.
Since $(F_k \circ \cdots \circ F_0)(x) = 0$ for some $k \in \Nat$, 
it follows that $(G_m \circ \cdots \circ G_0)(x) = 0$ for some $m \in \Nat$.
This shows that $x \in \Nil( \Delta_\circ )$ and hence that $\Nil( \Delta ) \subseteq  \Nil( \Delta_\circ )$.
By \eqref{9812hsj2k4gs93u}, we get  $\Nil( \Delta_\circ ) =  \Nil( \Delta )$.

We continue to consider $x \in \Nil(\Delta) = \Nil(\Delta_\circ)$.
Let $(H_0,H_1,\dots) \in \bar\Delta^\Nat$; let us prove that there exists $n$ such that $(H_n \circ \cdots \circ H_0)(x)=0$.
Recall that $\bar\Delta = \Span_\bk( \Delta_\circ )$; for each $i \in \Nat$, we choose 
$m_i \in \Nat$, $\lambda_{i,0}, \dots,\lambda_{i,m_i}  \in \bk$ and $G_{i,0}, \dots, G_{i,m_i} \in \Delta_\circ$ satisfying
\begin{equation} \label {9jdq90we120rfh8we}
\textstyle   H_i = \sum_{j=0}^{m_i}  \lambda_{i,j} G_{i,j} \, .
\end{equation}
The element $x$, the sequence $(H_0,H_1,\dots)$ and the relations \eqref{9jdq90we120rfh8we} determine
a directed tree $T$ that we now define.
The elements of the vertex-set $S$ are all finite sequences $(G_{0,j_0}, G_{1,j_1}, \dots, G_{n,j_n} )$ such that  $(G_{n,j_n} \circ \cdots \circ G_{0,j_0} )(x) \neq 0$
(where  $0 \le j_i \le m_i$ for all $i = 0, \dots, n$ and where the $G_{i,j}$ are the same as in \eqref{9jdq90we120rfh8we}).
The edge-set of $T$ is the subset $E$ of $S \times S$ whose elements are the pairs of the form
$\big( (G_{0,j_0}, \dots, G_{n,j_n} ), \, (G_{0,j_0}, \dots, G_{n,j_n}, G_{n+1,j_{n+1}} ) \big)$.
The root of $T$ is the empty sequence $() \in S$.
Note that if 
$\big(
(), \, 
( G_{0,j_0} ), \, 
( G_{0,j_0}, G_{1,j_1} ), \, 
( G_{0,j_0}, G_{1,j_1}, G_{2,j_2} ), \, \dots 
\big)$
is an infinite path in $T$ then the sequence $( G_{0,j_0}, G_{1,j_1}, G_{2,j_2}, \dots ) \in \Delta_\circ^\Nat$ satisfies
$(G_{n,j_n} \circ \cdots \circ G_{0,j_0} )(x) \neq 0$ for all $n \in \Nat$,
which contradicts the fact that  $x \in \Nil(\Delta_\circ)$.  So, {\it there is no infinite path in $T$ starting at the root.}
We also note that if $v = (G_{0,j_0}, \dots, G_{n,j_n} )$ is any vertex then the number of vertices $v' \in S$
satisfying $(v,v') \in E$ is at most $m_{n+1}$, which is finite.
It follows from Thm \ref{0veohtg09ernd0} that $T$ has finitely many vertices.
So there exists $n \in \Nat$ satisfying:
\begin{equation} \label {2830fuhe09323wsdf}
(G_{n,j_n} \circ \cdots \circ G_{0,j_0} )(x) = 0 \quad \text{for all $(j_0, \dots, j_n) \in \II_{m_0} \times \cdots \times \II_{m_n}$}
\end{equation}
where we write  $\II_m = \setspec{i \in \Nat}{ 0 \le i \le m }$. 
Let $n$ be as in \eqref{2830fuhe09323wsdf} and note that $(H_n \circ \cdots \circ H_0)(x)$ is a linear combination 
\begin{equation} \label {dpcenrb02923b0md}
\sum_{(j_0, \dots, j_n) \in \II_{m_0} \times \cdots \times \II_{m_n}} \mu_{(j_0, \dots, j_n)} (G_{n,j_n} \circ \cdots \circ G_{0,j_0} )(x) 
\end{equation}
where $\mu_{(j_0, \dots, j_n)} \in \bk$ for all $(j_0, \dots, j_n) \in \II_{m_0} \times \cdots \times \II_{m_n}$.
By  \eqref{2830fuhe09323wsdf}, all terms of the sum \eqref{dpcenrb02923b0md} are zero, so $(H_n \circ \cdots \circ H_0)(x) = 0$.
This proves that $x \in \Nil( \bar \Delta )$ and hence that  $\Nil( \Delta ) \subseteq  \Nil( \bar \Delta )$.
Then \eqref{9812hsj2k4gs93u} gives $\Nil( \bar \Delta ) = \Nil( \tilde \Delta ) = \Nil( \Delta )$,
so (a) is proved.

Let us prove (b).
If $(G_1, \dots, G_n) \in \Delta_\circ^n$ then $G_n \circ \cdots \circ G_1 = F_N \circ \cdots \circ F_1$ for some $N\ge n$ and $(F_1, \dots, F_N) \in \Delta^N$.
This implies that $\deg_{ \Delta_\circ }(x) \le \deg_{ \Delta }(x)$ for all $x \in V$, so we obtain
$\deg_{ \Delta_\circ } = \deg_{ \Delta }$ by \eqref{93hfd54e15syu3rtdj}.

Consider $x \in V \setminus \{0\}$ and let us prove that
$\deg_{\bar\Delta}(x) \le \deg_{\Delta_\circ}(x)$. We may assume that $\deg_{\Delta_\circ}(x) = n \in \Nat$. 
If $(H_1, \dots, H_{n+1}) \in \bar\Delta^{n+1}$ then $(H_{n+1} \circ \cdots \circ H_1)(x)$ is a finite sum $\sum_i \lambda_i t_i$
where $\lambda_i \in \bk$ and $t_i = (G_{i,n+1} \circ \cdots \circ G_{i,1})(x)$ with $G_{i,1},\dots,G_{i,n+1} \in \Delta_\circ$.
Since $\deg_{\Delta_\circ}(x) = n$ we have $t_i=0$ for all $i$, so $(H_{n+1} \circ \cdots \circ H_1)(x) = 0$.
Consequently, $\deg_{\bar\Delta}(x) \le n = \deg_{\Delta_\circ}(x)$.
We get $\deg_{\bar\Delta} = \deg_{\Delta_\circ}$ by \eqref{93hfd54e15syu3rtdj},
and since $\deg_{ \Delta_\circ } = \deg_{ \Delta }$ it follows that $\deg_{\bar\Delta} = \deg_{\Delta}$.
Using \eqref{93hfd54e15syu3rtdj} again gives
$\deg_{\Delta} = \deg_{\tilde\Delta} = \deg_{\bar\Delta}$, and hence $\UNil(\Delta) = \UNil(\tilde\Delta) = \UNil(\bar\Delta)$.
So (b) is proved.
Assertions (c) and (d) follow from (a) and (b).
\qed\end{proof}

\begin{corollary}
Let $A$ be either an associative subalgebra of $\Leul_\bk(V)$ or a Lie subalgebra of $\Leul_\bk(V)_\LL$. 
If $A$ is finitely generated then $\Nil(A) = \UNil(A)$ and, consequently,
$A$ is a locally nilpotent subset of $\Leul_\bk(V)$ if and only if it is 
a uniformly locally nilpotent subset of $\Leul_\bk(V)$.
\end{corollary}

\begin{proof}
Let $\Delta$ be a finite generating set for the algebra $A$.
Then $\Nil(A) = \Nil(\Delta) = \UNil(\Delta) = \UNil(A)$ by Prop.\ \ref{d0c1i2eb9ew03j} and Lemma \ref{0cccno12q9wdnp}.
\qed\end{proof}

\section*{A special case: sets of derivations}

Let $(A,\cdot)$ be an algebra over a field $\bk$, in the sense of \ref{0bg93mm5k6gnndrtuiecj8}(a).
Recall (from \ref{p0vQ398vf83ms43rt7uxcy9w47F}) that this determines a Lie subalgebra $\Der_\bk(A)$ of $\Leul_\bk(A)_\LL$.
If $\Delta$ is a subset of  $\Der_\bk(A)$ then $\Delta \subseteq \Leul_\bk(A)$, so it makes sense to consider
the subsets $\Nil(\Delta)$ and $\UNil(\Delta)$ of $A$, and the map $\deg_\Delta : A \to \Nat \cup \{ \pm \infty \}$.

The rest of the section is devoted to the proof of:

\begin{theorem}  \label {cp09vv3b49fooOi1wqk498fb}
Let $(A,\cdot)$ be an algebra over a field $\bk$. 
If $\Delta$ is a subset of $\Der_\bk(A)$ then the following hold.
\begin{enumerata}

\item The sets $\Nil(\Delta)$ and $\UNil(\Delta)$ are subalgebras of $(A,\cdot)$.

\item The map $\deg_\Delta : A \to \Nat \cup \{ \pm\infty \}$ satisfies
$\deg_\Delta(x\cdot y) \le \deg_\Delta(x) + \deg_\Delta(y)$ for all $x,y \in \UNil(\Delta)$.

\end{enumerata}
\end{theorem}

We already know that $\Nil(\Delta)$ and $\UNil(\Delta)$ are linear subspaces of $A$.
So, to prove (a), it suffices to show that those two sets are closed under the multiplication of $A$. 

\begin{Remark} \label {0f9345eo78u7884r}
If $\bk \subseteq A$ then  $\bk \subseteq \UNil(\Delta) \subseteq\Nil(\Delta)$, because $\bk \subseteq \ker D$ for all $D \in \Der_\bk(A)$.
\end{Remark}

Until the end of the section, we assume that $(A,\cdot)$ is a $\bk$-algebra.
The Theorem follows from Lemmas \ref{09fo1i2wdjosd} and \ref{93rrf03bfv0d} and Cor.\ \ref{pc90b345tf7w938nf}.

\begin{notation}
If $n \in \Nat$, we write $\II_n = \setspec{i \in \Nat}{ 0 \le i \le n }$.
Let $n \in \Nat$ and $F_0, \dots, F_n \in \Leul_\bk(A)$.
If $I = \{ i_1, \dots, i_m \}$ is a subset of $\II_n$ and $i_1 < \cdots < i_m$, we write
$F_I = F_{i_m} \circ \cdots \circ F_{i_1}$ (if $I = \emptyset$ then $F_I = \id_A$).
If $S$ is a set then $\powerset(S)$ denotes the set of all subsets of $S$.
\end{notation}

\begin{lemma} \label {823y292ehfwej}
Let $x,y \in A$, $n \in \Nat$  and $D_0, \dots, D_n \in \Der_\bk(A)$.  Then
$$
(D_n \circ \cdots \circ D_0)(x\cdot y) = \sum_{I \in \smallpowerset(\II_n)} D_I(x) \cdot  D_{\II_n\setminus I}(y) .
$$
\end{lemma}

\begin{proof}
This can be proved by induction on $n$.
The straightforward argument is left to the reader.
\qed\end{proof}

\begin{Definition}
We define a partial order $\preceq$ on the set $\powerfin(\Nat)$ of finite subsets of $\Nat$
by declaring that the condition $I \preceq J$ (where $I,J \in \powerfin(\Nat)$) is equivalent to:
\begin{center}
$I \subseteq J$ and for every $i \in I$ and $j \in J \setminus I$ we have $i<j$.
\end{center}
Note in particular that $\emptyset \preceq J$ for all $J \in \powerfin(\Nat)$.
Also observe that if $J \in \powerfin(\Nat)$ and $m \in \Nat$ then $\II_m \cap J \preceq J$.
\end{Definition}

\begin{lemma} \label {9237e8di9j3w8}
Let $X$ be a subset of $\powerfin(\Nat)$ such that 
\begin{itemize}

\item for all $I,J \in \powerfin(\Nat)$ satisfying $I \preceq J$ and $J \in X$, we have $I \in X$;

\item $(X,\preceq)$ satisfies the ascending chain condition (ACC).

\end{itemize}
Then the set
$$
Z = \setspec{ n \in \Nat }{
\text{there exists a subset $I$ of $\II_n$ such that $I \in X$ and $\II_n \setminus I \in X$} }
$$
is finite.
\end{lemma}

\begin{proof}
By contradiction, assume that $Z$ is an infinite set.
For each $n \in Z$, choose $I_n \in \powerset(\II_n)$ such that $I_n \in X$ and $\II_n \setminus I_n \in X$.
We inductively define an infinite sequence of sets $Z_0, Z_1, \dots$.
Choose an infinite subset $Z_0$ of $Z$ such that
$\big(\II_0 \cap I_n\big)_{n \in Z_0}$ is a constant sequence.
Let $m \in \Nat$ and suppose that $Z_0, \dots, Z_m$ satisfy:
\begin{itemize}

\item $Z \supseteq Z_0 \supseteq Z_1 \supseteq \cdots \supseteq Z_m$ are infinite sets;

\item  for each $j \in \{0, \dots, m\}$, 
$\min Z_j \ge j$ and $\big(\II_j \cap I_n\big)_{n \in Z_j}$ is a constant sequence.

\end{itemize}
Since  $\big(\II_{m+1} \cap I_n\big)_{n \in Z_m}$ is an infinite sequence of elements of the finite set 
$\powerset(\II_{m+1})$, we may choose an infinite subset $Z_{m+1}$ of $Z_m$ such that
$\big(\II_{m+1} \cap I_n\big)_{n \in Z_{m+1}}$ is a constant sequence and $\min Z_{m+1} \ge m+1$.
By induction, it follows that there exists an infinite sequence $Z_0, Z_1, Z_2, \dots$ satisfying:
\begin{enumerate}

\item[(i)] $Z \supseteq Z_0 \supseteq Z_1 \supseteq Z_2 \supseteq \cdots$ are infinite sets;
\item[(ii)]  for each $j \in \Nat$,
$\min Z_j \ge j$ and $\big(\II_j \cap I_n\big)_{n \in Z_j}$ is a constant sequence.

\end{enumerate}
For each $j \in \Nat$, define $B_j = \II_j \cap I_n$ for any $n \in Z_j$,
and let $B_j^*  = \II_j \setminus B_j$.
We shall now prove that the two sequences $\big( B_j \big)_{j \in \Nat}$ and $\big( B_j^*  \big)_{j \in \Nat}$ 
eventually stabilize; this is absurd, because $B_j \cup B_j^*  = \II_j$.

Let $j \in \Nat$.  Given any $n \in Z_{j}$ we have $B_j = \II_j \cap I_n \preceq I_n$ and $I_n \in X$,
so $B_j \in X$. Also, $\II_j \cap (\II_n \setminus I_n) \preceq \II_n \setminus I_n \in X$ implies
$\II_j \cap (\II_n \setminus I_n) \in X$; since
$\II_j \cap (\II_n \setminus I_n)
= (\II_j \cap \II_n)  \setminus (\II_j \cap I_n)
= \II_j \setminus B_j = B_j^* $
(where we used $n \in Z_j \Rightarrow n \ge j \Rightarrow \II_j \cap \II_n = \II_j$),
we have $B_j^*  \in X$. Thus
$\big( B_j \big)_{j \in \Nat}$ and $\big( B_j^*  \big)_{j \in \Nat}$ are sequences in $X$.

Again, let $j \in \Nat$.  Pick $n \in Z_{j+1}$ (so $n \in Z_j$), then
$\II_j \cap B_{j+1} = \II_j \cap (\II_{j+1} \cap I_n) = \II_j \cap I_n = B_j$, so $B_j \preceq B_{j+1}$.
Thus $B_0 \preceq B_1 \preceq B_2 \preceq \cdots$, and since $(X,\preceq)$ satisfies ACC
there exists $M \in \Nat$ such that  $\big( B_j \big)_{j \ge M}$ is a constant sequence.
For $j \ge M$ we have $B_j = B_{j+1}=B_M$, so
$\II_j \cap B_{j+1}^* 
= \II_j \cap (\II_{j+1} \setminus B_{j+1})
= \II_j \cap (\II_{j+1} \setminus B_M)
= \II_j \setminus B_M
= \II_j \setminus B_j = B_j^* $, so $B_j^*  \preceq B_{j+1}^* $.
Thus $B_M^*  \preceq B_{M+1}^*  \preceq \cdots$, so the sequence $\big( B_j^*  \big)_{j \ge M}$ stabilizes.
As we already explained, this is absurd.
\qed\end{proof}

Recall that $A$ is an algebra over a field $\bk$.

\begin{lemma}  \label {09fo1i2wdjosd}
For every subset $\Delta$ of $\Der_\bk(A)$, $\Nil(\Delta)$ is closed under the multiplication of $A$.
\end{lemma}

\begin{proof}
We may assume that $\Delta \neq \emptyset$.
Let $x,y \in \Nil(\Delta)$ and $(D_0, D_1, \dots ) \in \Delta^\Nat$.
To prove the Lemma, it suffices to show that there exists $n \in \Nat$ such that $(D_n \circ \cdots \circ D_0)(x\cdot y)=0$.

Define $X_x = \setspec{ I \in \powerfin(\Nat) }{ D_I(x) \neq 0 }$. Let us prove:
\begin{enumerate}

\item[(i)] for all $I,J \in \powerfin(\Nat)$ satisfying $I \preceq J$ and $J \in X_x$, we have $I \in X_x$;

\item[(ii)] $(X_x , \preceq)$ satisfies the ACC.

\end{enumerate}
Consider $I,J \in \powerfin(\Nat)$ satisfying $I \preceq J$ and $J \in X_x$.
Then $0 \neq D_J(x) = (D_{J\setminus I} \circ D_I)(x)$, 
so $D_I(x) \neq 0$, so $I \in X_x$.  This proves (i).

Consider an infinite sequence $I_0 \preceq I_1 \preceq \cdots$ in $X_x$. By contradiction, assume that
$\big( I_n \big)_{n \in \Nat}$ does not stabilize. Then $I = \bigcup_{n \in \Nat} I_n$ is  an infinite subset of $\Nat$.
Let $i_0 < i_1 < i_2 < \cdots$ be the elements of $I$
and consider $(D_{i_0}, D_{i_1}, \dots) \in \Delta^\Nat$.
Given any $k \in \Nat$ there exists $n \in \Nat$ such that $\{ i_0 , i_1 , \dots, i_k \} \subseteq I_n$;
then  $\{ i_0 , i_1 , \dots, i_k \} \preceq I_n \in X_x$, so 
$\{ i_0 , i_1 , \dots, i_k \} \in X_x$ by (i), so $(D_{i_k} \circ \cdots \circ D_{i_0})(x) \neq 0$.
Since this holds for all $k \in \Nat$, we have $x \notin \Nil(\Delta)$, a contradiction.
So (ii) is proved.

Let $X_y = \setspec{ I \in \powerfin(\Nat) }{ D_I(y) \neq 0 }$; then it is clear that (i) and (ii) are true with ``$X_y$'' in place of ``$X_x$''.
Consequently, if we define $X = X_x \cup X_y$ then:
\begin{enumerate}

\item[(iii)] for all $I,J \in \powerfin(\Nat)$ satisfying $I \preceq J$ and $J \in X$, we have $I \in X$;

\item[(iv)] $(X , \preceq)$ satisfies the ACC.

\end{enumerate}
By (iii), (iv) and Lemma \ref{9237e8di9j3w8}, the set
$$
Z =\setspec{ n \in \Nat }{
\text{there exists a subset $I$ of $\II_n$ such that $I \in X$ and $\II_n \setminus I \in X$} }
$$
is finite. Choose $n \in \Nat \setminus Z$. Then for each 
subset $I$ of $\II_n$ we have $I \notin X$ (i.e., $D_I(x)=0=D_I(y)$) or $\II_n \setminus I \notin X$
(i.e., $D_{\II_n\setminus I}(x)=0=D_{\II_n\setminus I}(y)$).
So all terms are zero in the right-hand-side of
$$
(D_n \circ \cdots \circ D_0)(x\cdot y) = \sum_{I \in \smallpowerset(\II_n)} D_I(x) \cdot  D_{\II_n\setminus I}(y) 
$$
(the equality follows from Lemma \ref{823y292ehfwej}).
Thus $(D_n \circ \cdots \circ D_0)(x\cdot y) = 0$, as desired.
\qed\end{proof}

\begin{lemma}  \label {93rrf03bfv0d}
Let $\Delta$ be a subset of $\Der_\bk(A)$. Then the map $\deg_\Delta : A \to \Nat \cup \{ \pm\infty \}$ satisfies
$\deg_\Delta(x\cdot y) \le \deg_\Delta(x) + \deg_\Delta(y)$ for all $x,y \in \UNil(\Delta)$.
\end{lemma}

\begin{proof}
The case $\Delta = \emptyset$ is trivial, so let us assume that $\Delta \neq \emptyset$.
It suffices to show that $P(n)$ is true for all $n \in \Nat$, where we define
\begin{equation*}
P(n): \qquad \forall_{x,y \in A\setminus\{0\}}\ 
\big(  \deg_\Delta(x) + \deg_\Delta(y) \le n \implies \deg_\Delta(x\cdot y) \le n  \big).
\end{equation*}

Suppose that $x,y \in A\setminus\{0\}$  satisfy $\deg_\Delta(x) + \deg_\Delta(y) \le 0$.
Then  $\deg_\Delta(x) = 0 = \deg_\Delta(y)$, so for all $D \in \Delta$ we have $D(x)=0=D(y)$ and hence $D(x\cdot y) = D(x)\cdot y + x\cdot D(y) = 0$,
so $\deg_\Delta(x\cdot y) \le 0$, showing that $P(0)$ is true.

Let $n \in \Nat$ and assume that $P(n)$ is true.
To prove that $P(n+1)$ is true, consider $x,y \in A\setminus\{0\}$  satisfying
$\deg_\Delta(x) + \deg_\Delta(y) \le n+1$.
We have to show that $\deg_\Delta(x\cdot y) \le n+1$.
To do this, it suffices to show that $(D_{n+2} \circ \cdots \circ D_1)(x\cdot y)=0$ for all $(D_1, \dots, D_{n+2}) \in \Delta^{n+2}$.
So let us consider  $(D_1, \dots, D_{n+2}) \in \Delta^{n+2}$.
As a first step, we claim that 
\begin{equation}  \label {iIx9c23rg093n629kjkEews4I}
\deg_\Delta( D_1(x) \cdot  y ) \le n \quad \text{and} \quad \deg_\Delta( x \cdot  D_1(y) ) \le n .
\end{equation}
To see this, first note that
if $D_1(x) = 0$ then $\deg_\Delta( D_1(x) \cdot  y ) = -\infty \le n$;
so, to prove the first part of \eqref{iIx9c23rg093n629kjkEews4I}, we may assume that $D_1(x) \neq 0$.
Note that $\deg_\Delta( D_1(x) ) < \deg_\Delta(x)$ by part (d) of Lemma \ref{kjcoo93edg98uwb},
so $\deg_\Delta( D_1(x) ) + \deg_\Delta(y) \le n$; since both $D_1(x)$ and $y$ belong to $A\setminus\{0\}$
and $P(n)$ is true, we get $\deg_\Delta( D_1(x) \cdot  y ) \le n$.
By the same argument we have $\deg_\Delta( x \cdot  D_1(y) ) \le n$, so \eqref{iIx9c23rg093n629kjkEews4I} is proved.
Then
\begin{align} \label {0c9v2b3498fvb0}
\deg_\Delta( D_1(x\cdot y) )
&= \deg_\Delta( D_1(x) \cdot  y + x \cdot  D_1(y) ) \\
\notag &\le \max \big( \deg_\Delta( D_1(x) \cdot  y ) ,  \deg_\Delta( x \cdot  D_1(y) ) \big) \le n
\end{align}
by part (e) of Lemma \ref{kjcoo93edg98uwb} and \eqref{iIx9c23rg093n629kjkEews4I}.
Now \eqref{0c9v2b3498fvb0} implies that
$$
(D_{n+2} \circ \cdots \circ D_1)(x\cdot y) = (D_{n+2} \circ \cdots \circ D_2)(D_1(x\cdot y)) =0,
$$
which is the desired conclusion.  So $\deg_\Delta(x\cdot y) \le n+1$ and hence $P(n+1)$ is true.
\qed\end{proof}

\begin{corollary}  \label {pc90b345tf7w938nf}
Let $\Delta$ be a subset of $\Der_\bk(A)$. Then $\UNil(\Delta)$ is closed under the multiplication of~$A$.
\end{corollary}

\begin{proof}
This follows from Lemma \ref{93rrf03bfv0d} and Part (c) of Lemma \ref{kjcoo93edg98uwb}.
\qed\end{proof}

This completes the proof of Thm \ref{cp09vv3b49fooOi1wqk498fb}.

\section{Lie-locally nilpotent sets of linear maps}
\label {SEC:LieLocallynilpotentsetsoflinearmaps}

Let $V$ be a vector space over a field $\bk$.

\begin{caution}
Let $x \in V$ and $F_0, \dots, F_{n+1} \in \Leul_\bk(V)$.
It is clear that 
$(F_n \circ \cdots \circ F_0)(x) = 0$ implies $(F_{n+1} \circ F_n \circ \cdots \circ F_0)(x) = 0$ 
and that $[F_n, \dots, F_0] = 0$ implies $[F_{n+1},F_n, \dots, F_0] = 0$,
but it is not the case that 
$[F_n, \dots, F_0](x) = 0$ implies $[F_{n+1},F_n, \dots, F_0](x) = 0$.
(This should be kept in mind while reading this Section.)
\end{caution}

\begin{Definition} \label {Lie0f9v358f6c7we9ww5gr}
Given a subset $\Delta$ of $\Leul_\bk(V)$, we define the subsets $\Nil^\LL(\Delta)$ and $\UNil^\LL(\Delta)$ of $V$ as follows.
If $\Delta=\emptyset$, we set $\Nil^\LL(\Delta)=V=\UNil^\LL(\Delta)$.
If $\Delta \neq \emptyset$,
\begin{itemize}

\item $\Nil^\LL(\Delta)$ is the set of all $x \in V$ satisfying:
\begin{quote}
for each $(F_0,F_1,\dots) \in \Delta^\Nat$ there exists $N \in \Nat$ such that for all $m\ge0$ and $(G_1,\dots,G_m) \in \Delta^m$
we have \text{$[G_m, \dots, G_1, F_N , \dots, F_0](x) =0$.}\footnote{We compute
$[G_m, \dots, G_1, F_N , \dots, F_0]$ in the Lie algebra $\Leul_\bk(V)_\LL$ (see \ref{p0vQ398vf83ms43rt7uxcy9w47F} and \ref{pa9enbr2o39efcp230wsd}).}
\end{quote}

\item $\UNil^\LL(\Delta)$ is the set of all $x \in V$ satisfying:
\begin{quote}
there exists $N \in \Nat$ such that for every $n \ge N$ and $(F_1,\dots,F_n) \in \Delta^n$ we have 
$[F_n, \dots, F_1](x) =0$.
\end{quote}

\end{itemize}
Note that $\UNil^\LL(\Delta) \subseteq \Nil^\LL(\Delta)$ are linear subspaces of $V$.
If $\Nil^\LL(\Delta) = V$, we say that the set $\Delta$ is \textit{Lie-locally nilpotent};
if $\UNil^\LL(\Delta) = V$, we say that $\Delta$ is \textit{uniformly Lie-locally nilpotent}.
\end{Definition}

\begin{Remark} \label {pckv349thndd8d539W}
The condition that $\Delta$ is a Lie-locally nilpotent (or a uniformly Lie-locally nilpotent) subset of $\Leul_\bk(V)$
does not imply that $\Delta \subseteq \locnil( \Leul_\bk V )$.
For instance, if $F \in \Leul_\bk(V) \setminus \locnil(\Leul_\bk V)$ then $\Delta=\{F\}$ is Lie-locally nilpotent.
\end{Remark}

\begin{Example}  \label {ckj029330fjbsxdkk56930mf}
In general there are no inclusion relations between $\Nil(\Delta)$ and $\Nil^\LL(\Delta)$.
The $\Delta$ of \ref{pckv349thndd8d539W} satisfies  $\Nil(\Delta) \nsupseteq \Nil^\LL(\Delta)$.
We now give an example satisfying  $\Nil(\Delta) \nsubseteq \Nil^\LL(\Delta)$.

Let  $(e_i)_{i \in \Nat}$ be a basis of a vector space $V$ over a field $\bk$.
Let $\Delta = \{ F_1, F_2, F_3, \dots \}$ where for each $n \ge 1$ we define the $\bk$-linear map $F_n : V \to V$ by
$$
\text{$F_n(e_i) = 0$ for all $i \in \{1, \dots, n\}$} \qquad \text{and} \qquad \text{$F_n(e_i) = e_n$ for all $i \in \Nat \setminus \{1, \dots, n\}$.}
$$
Then we have:
\begin{equation} \label {c0viuhh348d78bxcpwe9J}
F_n \circ F_m = 0 \quad \text{for all $m,n$ such that $m \le n$.}
\end{equation}

We claim that $\Nil(\Delta)=V$.
Indeed, consider a sequence $(F_{i_0}, F_{i_1}, F_{i_2}, \dots ) \in \Delta^\Nat$.
Then $(i_0,i_1,i_2,\dots)$ cannot be strictly decreasing, so there must exist $n\ge1$ such that $i_{n-1}\le i_n$.
Then $F_{i_n} \circ F_{i_{n-1}} = 0$ by \eqref{c0viuhh348d78bxcpwe9J}, so $F_{i_n} \circ \cdots \circ F_{i_0} = 0$.
This implies in particular that  $\Nil(\Delta)=V$.

Now consider the sequence $(F_1, F_2, F_3, \dots) \in \Delta^\Nat$.
Using \eqref{c0viuhh348d78bxcpwe9J}, it is easy to show by induction that
$[F_n, \dots, F_1] = (-1)^{n+1} F_1 \circ F_2 \circ \cdots F_n$ for all $n\ge1$.
Since $(F_1 \circ F_2 \circ \cdots F_n)(e_0) = e_1$ for all $n\ge1$, we have 
$[F_n, \dots, F_1](e_0) \neq 0$  for all $n\ge1$, so $e_0 \notin \Nil^\LL(\Delta)$.
So $\Nil(\Delta) \nsubseteq \Nil^\LL(\Delta)$.
\end{Example}

In contrast with Ex.\ \ref{ckj029330fjbsxdkk56930mf}, we have:

\begin{lemma} \label {p9vcuwbIhEj483u839}
For every subset $\Delta$ of $\Leul_\bk(V)$, we have $\UNil( \Delta ) \subseteq \UNil^\LL( \Delta )$.
\end{lemma}

\begin{proof}
Let $x \in \UNil( \Delta )$. 	
Then there exists $N\ge 1$ such that $(F_N \circ \cdots \circ F_1)(x) = 0$ for all $(F_1,\dots,F_N) \in \Delta^N$. So,
\begin{equation} \label {ocvihb394eowed-1}
(F_n \circ \cdots \circ F_1)(x) = 0 \quad \text{for all $n\ge N$ and all  $(F_1,\dots,F_n) \in \Delta^n$.}
\end{equation}
Now given any $n\ge N$ and $(G_1,\dots,G_n) \in \Delta^n$ we have 
\begin{equation*} 
[G_n, \dots, G_1] \in \Span_\bk \setspec{ F_n \circ \cdots \circ F_1 }{ (F_1,\dots,F_n) \in \Delta^n },
\end{equation*}
so $[G_n, \dots, G_1](x)=0$ by \eqref{ocvihb394eowed-1}.
This shows that $x \in \UNil^\LL( \Delta )$, as desired.
\qed\end{proof}

\begin{lemma} \label {Nil-0cccno12q9wdnp}
Let $\Delta$ be a finite subset of $\Leul_\bk(V)$. Then $\Nil^\LL(\Delta) = \UNil^\LL(\Delta)$.
In particular, $\Delta$ is Lie-locally nilpotent if and only if it is uniformly Lie-locally nilpotent.
\end{lemma}

\begin{proof}
It suffices to show that  $\Nil^\LL(\Delta) \subseteq \UNil^\LL(\Delta)$.
Let  $x \in \Nil^\LL(\Delta) \setminus \{0\}$.
Consider the set $S$ of finite sequences $(F_0,F_1 \dots, F_N)$ of elements of $\Delta$ 
for which there exist $m\ge0$ and $(G_1, \dots, G_m) \in \Delta^m$ satisfying $[G_m, \dots, G_1, F_N, \dots, F_0](x) \neq 0$.
Let $E \subseteq S \times S$ be the set of ordered pairs of the form
$$
\big( (F_0, \dots, F_n), (F_0, \dots, F_n, F_{n+1}) \big) \; .
$$
Then $(S,E)$ is a directed tree and the empty sequence $() \in S$ is the root of this tree (see \ref{09cn3049dEw}).
For each vertex $v \in S$, the number of $v'\in S$ satisfying $(v,v') \in E$ is at most $|\Delta|$, which is finite.
The fact that $x \in \Nil^\LL(\Delta)$ implies that there is no infinite path $\big( (), (F_0), (F_0,F_1), \dots \big)$ in this tree.
It follows from Thm \ref{0veohtg09ernd0} that $S$ is a finite set, and this implies that $x \in \UNil^\LL(\Delta)$.
\qed\end{proof}

\begin{proposition}  \label {Nil-d0c1i2eb9ew03j}
Let $\Delta$ be a subset of $\Leul_\bk(V)$
and let $\tilde\Delta$ be the Lie subalgebra of $\Leul_\bk(V)_\LL$ generated by $\Delta$.
Then the following hold.
\begin{enumerata}

\item Let $x \in V$. If $N$ is a positive integer satisfying
\begin{equation} \label {1ov9bn928bf90jehop}
[F_n, \dots, F_1](x) = 0 \quad \text{for all $n \ge N$ and all $(F_1,\dots,F_n) \in \Delta^n$},
\end{equation}
then the same $N$ satisfies
\begin{equation*} 
[F_n, \dots, F_1](x) = 0 \quad \text{for all $n \ge N$ and all $(F_1,\dots,F_n) \in \tilde\Delta^n$}.
\end{equation*}

\item $\UNil^\LL(\tilde\Delta) = \UNil^\LL(\Delta)$

\item If $\Delta$ is uniformly Lie-locally nilpotent then so is $\tilde\Delta$.

\end{enumerata}
\end{proposition}

\begin{proof}
(a) Let $x \in V$ and let $N$ be a positive integer satisfying \eqref{1ov9bn928bf90jehop}. Let $n\ge N$.
To prove (a), it suffices to show that
\begin{equation}  \label {1-Avn3eo9vwej092X8eiUY}
[H_n, \dots, H_1](x) = 0 \quad \text{for all $(H_1,\dots,H_n) \in \tilde\Delta^n$}.
\end{equation}
Consider the set
$ \Delta_{\square} = \setspec{ [F_k, \dots, F_1] }{ k \ge 1 \text{ and }  (F_1, \dots, F_k) \in \Delta^k } $
and note that $\Delta \subseteq \Delta_{\square} \subseteq \tilde\Delta$ and (by Lemma \ref{0vh349rf920}(b)) $\tilde\Delta = \Span_\bk( \Delta_\square )$.
The first step in the proof of \eqref{1-Avn3eo9vwej092X8eiUY} consists in proving
\begin{equation}  \label {2-Avn3eo9vwej092X8eiUY}
[G_n, \dots, G_1](x) = 0 \quad \text{for all $(G_1,\dots,G_n) \in (\Delta_{\square})^n$}.
\end{equation}
Let $(G_1,\dots,G_n) \in (\Delta_{\square})^n$.
For each $i \in \{1, \dots, n\}$, write 
$$
G_i = [F_{i,k_i}, \dots, F_{i,2}, F_{i,1}]
$$
where $k_i\ge1$ and $F_{i,j} \in \Delta$.
Then Lemma \ref{0vh349rf920}(a) implies that $[G_n, \dots, G_1] = \sum_i \lambda_i [W_{i,k}, \dots, W_{i,1}]$ where $k = k_1+\cdots+k_n \ge n$
and (for each $i$) $\lambda_i \in \bk$ and $(W_{i,1}, \dots, W_{i,k})$ is a permutation of $(F_{1,1}, \dots, F_{1,k_1}, F_{2,1}, \dots, F_{n,k_n}) \in \Delta^k$. 
Since $k\ge n \ge N$,  \eqref{1ov9bn928bf90jehop} implies that $[W_{i,k}, \dots, W_{i,1}](x)=0$ for all $i$,  
so $[G_n, \dots, G_1](x) = 0$, proving \eqref{2-Avn3eo9vwej092X8eiUY}.

Now let us prove \eqref{1-Avn3eo9vwej092X8eiUY}.  Let $(H_1,\dots,H_n) \in \tilde\Delta^n$.
For each $i$ we have $H_i = \sum_j \lambda_{i,j} G_{i,j}$ with 
$\lambda_{i,j} \in \bk$ and $G_{i,j} \in \Delta_\square$.
Thus $[H_n, \dots, H_1]$ is a linear combination of terms of the form $[G_{n,j_n}, \dots, G_{1,j_1}]$.
By \eqref{2-Avn3eo9vwej092X8eiUY}, $[G_{n,j_n}, \dots, G_{1,j_1}](x)=0$ for every choice of $(j_1,\dots,j_n)$,
so $[H_n, \dots, H_1](x)=0$.
This proves \eqref{1-Avn3eo9vwej092X8eiUY}, so (a) is proved. Assertions (b) and (c) follow from (a).
\qed\end{proof}

\begin{theorem}  \label {Lie-cp09vv3b49fooOi1wqk498fb}
Let $(A,\cdot)$ be an algebra over a field $\bk$. 
If $\Delta$ is a subset of $\Der_\bk(A)$ then 
the sets $\Nil^\LL(\Delta)$ and $\UNil^\LL(\Delta)$ are subalgebras of $(A,\cdot)$.
\end{theorem}

\begin{proof}
Consider a sequence $\mathbf{D} = (D_0,D_1,\dots) \in \Delta^\Nat$.
For each $N \in \Nat$, let $K_{N}(\mathbf{D})$ be the intersection of the kernels of the derivations
$[G_m, \dots, G_1, D_N, \dots, D_0]$, for all $m\ge0$ and $(G_1, \dots, G_m) \in \Delta^m$.
Since $K_{0}(\mathbf{D}) \subseteq K_{1}(\mathbf{D}) \subseteq K_{2}(\mathbf{D}) \subseteq \cdots$ are subalgebras of $A$,
so is $K(\mathbf{D}) = \bigcup_{N=0}^\infty K_{N}(\mathbf{D})$.
Then $\Nil^\LL(\Delta) = \bigcap_{\mathbf{D} \in \Delta^\Nat} K( \mathbf{D} )$  is a subalgebra of $A$.

For each $N \in \Nat$, let $K_N$ be the intersection of the $\ker[D_n, \dots, D_1]$ for all $n\ge N$ and  $(D_1,\dots,D_n) \in \Delta^n$.
Then $K_0 \subseteq K_1 \subseteq K_2 \subseteq \cdots$ is a chain of subalgebras of $A$, so 
$\UNil^\LL(\Delta) = \bigcup_{N=0}^\infty K_N$ is a subalgebra of $A$.
\qed\end{proof}

\section{Nilpotency conditions for algebras}
\label {SEC:Someapplicationstoalgebras}

Throughout paragraphs \ref{Qpc0P9ivb29340edqows}--\ref{ovp9e83bih39w0edpckwm},
$(A,\cdot)$ is an algebra over a field $\bk$ in the sense of \ref{0bg93mm5k6gnndrtuiecj8}(a).
Starting in \ref{covp32094efdvnpwWzZ}, we restrict our attention to the special case where $A$ is either an associative or a Lie algebra.

\begin{Definition} \label {Qpc0P9ivb29340edqows}
Given a subset $H$ of $A$,
let us define the subsets $\Nil'(H)$ and $\UNil'(H)$ of $A$ as follows.
If $H=\emptyset$, we set  $\Nil'(H) = A = \UNil'(H)$. If $H \neq \emptyset$,
\begin{itemize}

\item $\Nil'(H)$ is the set of all $x \in A$ satisfying:
\begin{quote}
for every sequence $(a_0,a_1,\dots) \in H^\Nat$ there exists $n \in \Nat$ such that $a_n \cdots a_0 \cdot x =0$;\footnote{Recall
that $a_n \cdots a_0 \cdot x = 0$ means $a_n \cdot ( a_{n-1} \cdots ( a_1 \cdot (a_0 \cdot x)) \dots)=0$,
by the right-associativity convention (\ref{p0vQ398vf83ms43rt7uxcy9w47F}).}
\end{quote}

\item $\UNil'(H)$ is the set of all $x \in A$ satisfying:
\begin{quote}
there exists $n \in \Nat$ such that for every $(a_1,\dots,a_n) \in H^n$ we have $a_n \cdots a_1 \cdot x =0$.
\end{quote}

\end{itemize}
Define $\deg'_H : A \to \Nat \cup \{ \pm\infty\}$ by declaring that
$\deg'_H(0) = -\infty$,
$\deg'_H(x) = \infty$ for all $x \in A \setminus \UNil'(H)$, and if $x \in \UNil'(H) \setminus \{0\}$ then
$\deg'_H(x)$ is
$$
\max\setspec{ n \in \Nat }{ \text{there exists $(a_1, \dots, a_n) \in H^n$ satisfying $a_n \cdots a_1 \cdot x \neq 0$} } \; .
$$
If $H = \emptyset$ then we set $\deg'_H(0) = -\infty$ and $\deg'_H(x) = 0$ for all $x \in A \setminus \{0\}$.
\end{Definition}

Let us now connect these notions to the formalism of Section \ref{SEC:Locallynilpotentactions}.

\begin{nothing*}  \label {c0b78h6mrhfwnvbrfvWkdtr8230}
Define the map $\phi : A \to \Leul_\bk(A)$ by stipulating that, given $a \in A$, $\phi(a) : A \to A$ is the  $\bk$-linear map $x \mapsto a \cdot x$.
Note that $\phi$ is a $\bk$-linear map, but not necessarily a homomorphism of algebras.
If $H$ is a subset of $A$ then $\phi(H)$ is a subset of $\Leul_\bk(A)$, so we may consider (as in Sec.~\ref{SEC:Locallynilpotentactions}) the subsets
$\Nil( \phi(H) )$ and $\UNil( \phi(H) )$ of $A$ and the map $\deg_{\phi(H)} : A \to \Nat\cup\{ \pm\infty \}$.
Then we have
\begin{equation}  \label {c0v93h48r8g0q0wk}
\text{$\Nil'(H) = \Nil( \phi(H) )$, $\UNil'( H ) = \UNil( \phi(H) )$ and $\deg'_H = \deg_{\phi(H)}$.}
\end{equation}
Indeed, the right-associativity convention implies that $a_n \cdots a_0 \cdot x = \big( \phi(a_n) \circ \cdots \circ \phi(a_0) \big)(x)$
for all $x,a_0,\dots,a_n \in A$, and \eqref{c0v93h48r8g0q0wk} immediately follows.
\end{nothing*}

\begin{corollary} \label {pv9wge98epq0djzwe5r}
If $H$ is a finite subset of $A$ then $\Nil'(H) = \UNil'( H )$.
\end{corollary}

\begin{proof}
Apply Lemma \ref{0cccno12q9wdnp} to the finite set $\phi(H)$ and use \eqref{c0v93h48r8g0q0wk}.
\qed\end{proof}

\begin{lemma} \label {ovp9e83bih39w0edpckwm}
The sets $\Nil'(A)$ and $\UNil'(A)$ are left ideals of $A$.
Moreover, for each $n \in \Nat$ the set $Z_n = \setspec{ x \in A }{ \deg'_A(x) < n }$ is a left ideal of $A$.
\end{lemma}

\begin{proof}
We have $\{0\} = Z_0 \subseteq Z_1  \subseteq Z_2 \subseteq \cdots$ and 
$\UNil'(A) = \bigcup_{n=0}^\infty Z_n$; so, to prove the Lemma, it suffices to show that $\Nil'(A)$ and all $Z_n$ are left ideals of $A$.
It is clear that $\Nil'(A)$ and $Z_n$ are $\bk$-subspaces of $A$.

Let $x \in A$ and $y \in \Nil'(A)$. To prove that $x\cdot y \in \Nil'(A)$, we have to show that for every
$(x_0,x_1,\dots) \in A^\Nat$ there exist an $n$ such that  $x_n \cdots x_0 \cdot (x \cdot y) = 0$.
Note that  $x_n \cdots x_0 \cdot (x \cdot y) = x_n \cdots x_0 \cdot x \cdot y$ by the right-associativity convention.
So the claim is true, because $(x, x_0,x_1,\dots) \in A^\Nat$ and $y \in \Nil'(A)$.
So $x\cdot y \in \Nil'(A)$, showing that $\Nil'(A)$ is a left ideal of $A$.

Note that $Z_0 = \{0\}$ is indeed a left ideal of $A$. Let $n>0$, $x \in A$ and $y \in Z_n$. 
Then $x_n \cdots x_1 \cdot y = 0$ for all $(x_1, \dots, x_n) \in A^n$.
In particular, $x_n \cdots x_2 \cdot x \cdot y = 0$  for all $(x_2, \dots, x_n) \in A^{n-1}$.
Since $x_n \cdots x_2 \cdot (x \cdot y) = x_n \cdots x_2 \cdot x \cdot y=0$, this shows that $x\cdot y \in Z_{n-1} \subseteq Z_n$.
So $Z_n$ is a left ideal of $A$.
\qed\end{proof}

Cor.\ \ref{pv9wge98epq0djzwe5r} and Lemma \ref{ovp9e83bih39w0edpckwm} are valid for all algebras $A$.
We can say more if we assume that $A$ is either associative or Lie.

\begin{observation}  \label {covp32094efdvnpwWzZ}
Let $A$ be a $\bk$-algebra and consider the map $\phi$ of \ref{c0b78h6mrhfwnvbrfvWkdtr8230}.
\begin{enumerata}

\item If $A$ is an associative algebra then $\phi : A \to \Leul_\bk(A)$ is a homomorphism of associative algebras.

\item If $A$ is a Lie algebra then $\phi: A \to \Leul_\bk(A)_\LL$ is a homomorphism of Lie algebras and
$\phi(A) \subseteq \Der_\bk(A)$.

\end{enumerata}
\end{observation}

\begin{proof}
Part (a) is clear and (b) is the same as \ref{pa9enbr2o39efcp230wsd}(i,ii) (we have $\phi = \ad$).
\qed\end{proof}

\begin{corollary}  \label {1c0vh2h03rfcw8evjq}
Let $(A,\cdot)$ be either an associative algebra or a Lie algebra, let $H$ be a subset of $A$ and let $\hat H$ be the 
subalgebra of $(A,\cdot)$ generated by $H$. 
\begin{enumerata}

\item $\Nil'(H) = \Nil'( \hat H )$,  $\UNil'( H ) = \UNil'( \hat H )$ and $\deg'_H = \deg'_{\hat H}$.

\item If $(A,\cdot)$ is a Lie algebra then $\Nil'(H)$ and $\UNil'(H)$ are Lie subalgebras of $A$ and
$\deg'_H( x \cdot y ) \le \deg'_H(x) + \deg'_H(y)$ for all $x,y \in \UNil'(H)$.

\end{enumerata}
\end{corollary}

\begin{proof}
(a) Let $\bk$ be the field over which $A$ is an algebra.
Consider the algebra $( \Leul_\bk(A), \star )$, where for $F,G \in \Leul_\bk(A)$ we define
$F \star G = F \circ G$  (resp.\ $F \star G = F \circ G - G \circ F$) if we are proving the case where $A$ is associative (resp.\ $A$ is Lie).
Let $\widehat{\phi(H)}$ be the subalgebra of $( \Leul_\bk(A), \star )$ generated by $\phi(H)$.
Since $\phi : (A,\cdot) \to ( \Leul_\bk(A), \star )$  is a homomorphism of algebras by Obs.\ \ref{covp32094efdvnpwWzZ},
we have $\widehat{\phi(H)} = \phi( \hat H )$, so
$
\Nil'(H) = \Nil(\phi(H)) = \Nil( \widehat{\phi(H)} ) = \Nil(\phi( \hat H )) = \Nil'( \hat H )
$,
where the second equality is Prop.\ \ref{d0c1i2eb9ew03j} and the first and last equalities are \eqref{c0v93h48r8g0q0wk}.
It is clear that $\UNil'( H ) = \UNil'( \hat H )$ and $\deg'_H = \deg'_{\hat H}$ follow by the same argument,
so (a) is proved.

(b) If $(A,\cdot)$ is a Lie algebra then $\phi(H) \subseteq \Der_\bk(A,\cdot)$ by Obs.\ \ref{covp32094efdvnpwWzZ},
so Thm \ref{cp09vv3b49fooOi1wqk498fb} implies that 
$\Nil(\phi(H))$ and $\UNil(\phi(H))$ are Lie subalgebras of $(A,\cdot)$ and that 
$\deg_{\phi(H)}( x \cdot y ) \le \deg_{\phi(H)}(x) + \deg_{\phi(H)}(y)$ for all $x,y \in \UNil(\phi(H))$.
In view of \eqref{c0v93h48r8g0q0wk}, this is the desired conclusion.
\qed\end{proof}

\begin{notation} \label {c09v498riq0vvEobap33ytdfc}
Let $A$ be either an associative algebra or a Lie algebra. Given a subset $H$ of $A$, define
$s(H) = \sup \setspec{ \deg'_H(x) }{ x \in H } \in \Nat \cup \{ \pm\infty \}$,
where we agree that $s(\emptyset) = -\infty$.
Note that if $H \subseteq \{0\}$ then $s(H) = -\infty$, and if 
$H \nsubseteq \{0\}$ then 
$s(H) = \sup\setspec{ n \in \Nat }{ x_n \cdots x_0 \neq 0 \text{ for some $(x_0,\dots,x_n) \in H^{n+1}$} }$.
\end{notation}

\begin{lemma}  \label {093498hbdfhrjd738}
Let $A$ be either an associative algebra or a Lie algebra, let $H$ be a subset of $A$ and let $\hat H$ be the 
subalgebra of $A$ generated by $H$.
\begin{enumerata}

\item If $H \subseteq \Nil'(H)$ then $\hat H \subseteq \Nil'(\hat H)$.
\item If $H \subseteq \UNil'(H)$ then $\hat H \subseteq \UNil'(\hat H)$.
\item $s(H) = s(\hat H)$

\end{enumerata}
\end{lemma}

\begin{proof}
(a) Assume that $H \subseteq \Nil'(H)$.
Cor.\ \ref{1c0vh2h03rfcw8evjq} gives $\Nil'(H) = \Nil'(\hat H)$,
so $H \subseteq \Nil'(\hat H)$ and hence $H \subseteq \hat H \cap \Nil'(\hat H)$. 
Applying Lemma \ref{ovp9e83bih39w0edpckwm} {\it to the algebra $\hat H$} shows that $\hat H \cap \Nil'(\hat H)$ is a left ideal of $\hat H$
and hence a subalgebra of $\hat H$; 
since that subalgebra contains $H$, it must then contain $\hat H$, so $\hat H \subseteq\Nil'(\hat H)$.
This proves (a). We obtain a proof of (b) by replacing $\Nil'$ by $\UNil'$ everywhere.

(c) It is clear that $s(H) \le s(\hat H)$ and that equality holds whenever $s(H) \notin\Nat$. Assume that $s(H) = n \in \Nat$.
Since $\deg'_H = \deg'_{\hat H}$ by Cor.\ \ref{1c0vh2h03rfcw8evjq},
we have $\deg'_{\hat H}(x) \le n$ for all $x \in H$.
Thus $H \subseteq Z_{n+1}$, where we define $Z_j = \setspec{ x \in \hat H }{ \deg_{\hat H}'(x) < j }$.
Lemma \ref{ovp9e83bih39w0edpckwm} implies that $Z_{n+1}$ is a left ideal (and hence a subalgebra) of $\hat H$,
so $\hat H \subseteq Z_{n+1}$. It follows that $s(\hat H) \le n$, so $s(\hat H) = s(H)$.
\qed\end{proof}

\begin{corollary}   \label {p9v8e674jedrbvysauogUbu}
Let $A$ be either an associative algebra or a Lie algebra and let $H$ be a generating set for $A$.
Then the following hold.
\begin{enumerata}

\item If for every sequence $(a_0,a_1,\dots) \in H^\Nat$ there exists an $n$ such that $a_n \cdots a_0 = 0$,
then  for every sequence $(a_0,a_1,\dots) \in A^\Nat$ there exists an $n$ such that $a_n \cdots a_0 = 0$.

\item Suppose that $n$ is a positive integer such that for every $(a_1,\dots,a_n) \in H^n$ we have $a_n \cdots a_1 = 0$.
Then we have $a_n \cdots a_1 = 0$ for all  $(a_1,\dots,a_n) \in A^n$.

\end{enumerata}
\end{corollary}

\begin{proof}
We have $\hat H = A$ by assumption, where $\hat H$ denotes the subalgebra of $A$ generated by $H$.
Part (a) asserts that $H \subseteq \Nil'(H) \Rightarrow \Nil'(A)=A$; since $\hat H = A$, the claim follows from Lemma \ref{093498hbdfhrjd738}(a).
Part (b) asserts that if $n$ satisfies $s(H)<n-1$, then $s(A)<n-1$;  
Lemma \ref{093498hbdfhrjd738}(c) implies that $s(H) = s(\hat H) = s(A)$, so the claim is true.
\qed\end{proof}

\begin{Definition}  \label {c09ecxalxdsemeyOoxjIzkevrv}
Let $(A,\cdot)$ be either an associative algebra or a Lie algebra. 
Consider the map $\phi : A \to \Leul_\bk(A)$ of \ref{c0b78h6mrhfwnvbrfvWkdtr8230}.
We say that 
\begin{itemize}

\item $A$ is \textit{nilpotent} (N) 
if there exists an $n \ge 1$ such that for all $(x_1,\dots,x_n) \in A^{n}$ we have  $x_n \cdots x_1 = 0$;

\item $A$ is \textit{sequentially nilpotent} (SN) if for every infinite sequence $(x_0,x_1,\dots)$ of elements of $A$,
there exists an $n\ge0$ such that $x_n \cdots x_0 = 0$;

\item $A$ is \textit{locally nilpotent} (LN) if every finitely generated subalgebra of $A$ is nilpotent;

\item $A$ is \textit{nil} (nil) if for each $x \in A$ the map $\phi(x) : A \to A$ is nilpotent;

\item $A$ is \textit{locally nil} (Lnil) if for each $x \in A$ the map $\phi(x) : A \to A$ is locally nilpotent.

\end{itemize}
\end{Definition}

\begin{Remark}  
We found nothing in the literature about (SN) or (Lnil); as far as we know, these two conditions have not been considered previously.
The definitions of (N), (LN) and (nil) given in Def.\ \ref{c09ecxalxdsemeyOoxjIzkevrv} are compatible with standard usage of terminology.
Regarding (nil), note the following.
\begin{enumerate}[{label={\normalfont(\arabic*)}}]

\item Prop.\ \ref{GHIEco9vn359r0b3ru6eww}(d) shows that an associative algebra $A$
satisfies {\rm(nil)} if and only if every element of $A$ is nilpotent (which is the standard meaning of ``nil'' for associative algebras).

\item Our definition of ``nil'' for Lie algebras is the standard one.
Note that a Lie algebra that is nil is also said to be \textit{Engel}.
However, an associative algebra $A$ is said to be  \textit{Engel} if the Lie algebra $A_\LL$ is Engel.
So ``Engel'' and ``nil'' have the same meaning for Lie algebras but not for associative algebras.

\end{enumerate}
\end{Remark}

\begin{Remark}  \label {cv9349v20yuj7eeje49rfvh}
It follows from Lemma \ref{0vh349rf920} that if an associative or Lie algebra is both finitely generated and nilpotent,
then it is finite dimensional (as a vector space over $\bk$). 
Consequently, if $A$ is (LN) then  every finitely generated subalgebra of $A$ is nilpotent {\it and finite dimensional}.
\end{Remark}

\begin{proposition} \label {GHIEco9vn359r0b3ru6eww}
Let $A$ be either an associative algebra or a Lie algebra.
\begin{enumerata}

\item We have the following implications for $A$:\quad \raisebox{4.5mm}{$\xymatrix@1@C=8pt@R=0pt{
&& \text{\rm (SN)} \ar@{=>}[r] &  \text{\rm (LN)}  \ar@{=>}[dd] \\
\text{\rm (N)} \ar@{=>}[rru] \ar@{=>}[rrd] \\
&&  \text{\rm (nil)} \ar@{=>}[r] & \text{\rm (Lnil)\,.}
}$}

\smallskip

\item If $A$ is finitely generated then {\rm (N) $\Leftrightarrow$ (SN) $\Leftrightarrow$ (LN)}.

\item If $A$ is finite dimensional then {\rm (LN) $\Leftrightarrow$ (SN) $\Leftrightarrow$ (N) $\Leftrightarrow$ (nil) $\Leftrightarrow$ (Lnil)}.

\item If $A$ is associative then {\rm (SN) $\Rightarrow$ \rm (nil)} and moreover
$$
\textit{{\rm(Lnil)} $\Leftrightarrow$ {\rm(nil)} $\Leftrightarrow$ every element of $A$ is nilpotent.}
$$

\end{enumerata}
\end{proposition}

\begin{proof}
(a), (SN) $\Rightarrow$ (LN).
Assume that $A$ satisfies (SN) and let $H$ be a finite subset of $A$; we have to show that $\hat H$ is nilpotent,
where  $\hat H$ denotes the subalgebra of $A$ generated by $H$.
The assumption that $A$ satisfies (SN) implies that $H \subseteq \Nil'(H)$.
Since $H$ is finite we have $\UNil'(H) = \Nil'(H)$ by Cor.\ \ref{pv9wge98epq0djzwe5r}, so  $H \subseteq \UNil'(H)$ and consequently
$\deg_H'(x)<\infty$ for each $x \in H$.
Since $H$ is finite, it follows that $s(H)<\infty$.
Then Lemma \ref{093498hbdfhrjd738} implies that $s(H) = s( \hat H )$, so $s( \hat H )<\infty$, so $\hat H$ is nilpotent. So $A$ satisfies (LN).

(LN) $\Rightarrow$ (Lnil).
Suppose that $A$ is (LN).
To show that $A$ is (Lnil), it suffices to show that for all $x,y \in A$ there exists an $n>0$ such that $\phi(x)^n(y)=0$.
Consider $x,y \in A$.
As $A$ is (LN), the subalgebra $A_0$ of $A$ generated by $\{x,y\}$ is nilpotent,
so there exists $n>0$ such that $x_n \cdots x_0 = 0$ for all $(x_0, \dots, x_n) \in A_0^{n+1}$,
i.e., $\big( \phi(x_n) \circ \cdots \circ \phi(x_1) \big)(x_0) = 0$ for all $(x_0, \dots, x_n) \in A_0^{n+1}$.
As $x,y \in A_0$ it follows that $\phi(x)^n(y)=0$. So  $A$ is (Lnil).

The implications (Lnil) $\Leftarrow$ (nil) $\Leftarrow$ (N) $\Rightarrow$ (SN) are clear, so (a) is proved.

\smallskip

If $A$ is finitely generated then (LN) $\Rightarrow$ (N) is clear, by definition of (LN); so (b) is clear.

\smallskip

Let us prove (d) before (c). Suppose that $A$ is associative.

If $A$ satisfies (Lnil) then, for every choice of $x,y \in A$, there exists $n$ such that $\phi(x)^n(y)=0$;
so  for every $x \in A$ there exists $n$ such that $x^{n+1}  = \phi(x)^n(x)=0$.
This shows that (Lnil) implies that  every element of $A$ is nilpotent.

Since $\phi : A \to \Leul_\bk(A)$ is a homomorphism of associative algebras, it is clear that if every element
of $A$ is nilpotent then every element of $\phi(A)$ is nilpotent, i.e., $A$ satisfies (nil).

This proves the equivalences in (d).
Now suppose that $A$ satisfies (SN). Given any $x \in A$, applying the property (SN) to the sequence $(x,x,x,\dots) \in A^\Nat$ implies
that there exists $n$ such that $x^n=0$. So (SN) implies that every element of $A$ is nilpotent, and the proof of (d) is complete.

\smallskip

(c)
By (a),  it suffices to show that (Lnil) $\Rightarrow$ (nil) $\Rightarrow$ (N) when $A$ is finite dimensional.

If $A$ is finite dimensional then every locally nilpotent linear map $A \to A$ is in fact nilpotent, so (Lnil) implies (nil).

For finite dimensional Lie algebras,
the fact that (nil) implies (N) is {\it Engel's theorem on abstract Lie algebras}, see page 36 of  \cite{JacobsonLieAlg}.

Any finite dimensional associative $\bk$-algebra $A$ is (isomorphic to) a subalgebra of a matrix algebra $M_n(\bk)$, for some $n$.
If $A$ satisfies (nil), then part (d) implies that every element of $A$ is a nilpotent matrix.
It is then well known that $A$ is a nilpotent algebra (see for instance Thm 1, page 33 of  \cite{JacobsonLieAlg}).
So (nil) implies (N) in the associative case as well, and this proves (c).
\qed\end{proof}

\begin{nothing*}
By Ex.\ \ref{928349rfer8d9}, (SN) does not imply (nil) for Lie algebras.
\end{nothing*}

\begin{nothing}  \label {covib34oepowd0ccj029}
{\rm It is known that (nil) does not imply (LN).
Indeed, the following statements are valid over an arbitrary field and are consequences of more precise results due to Golod
(\cite{Golod64}, \cite{Golod68}, cited in \cite[p.\ 6]{Kostrikin_1990_book} and \cite[Thm 6.2.9]{RowenBookVol2}): }
\begin{enumerata}

\item \label {p0c9n3o49fXvkj4tb9}
For each integer $m \ge2$, there exists an infinite dimensional associative algebra $A$ generated by $m$ elements
such that every $(m-1)$-subset of $A$ is nilpotent.\footnote{A subset $H$ of an associative algebra $A$ is \textit{nilpotent} if there 
exists $n$ such that $h_n \cdots h_1=0$ for all $(h_1, \dots, h_n) \in H^n$. By Cor.\ \ref{p9v8e674jedrbvysauogUbu}, $H$ is nilpotent if
and only if $\bar H$ is nilpotent, where $\bar H$ is the subalgebra of $A$ generated by $H$.}
Moreover, $\bigcap_{n=1}^\infty A^n = \{0\}$.

\item There exists a $3$-generated Lie algebra that satisfies \mbox{\rm(nil)} but not \mbox{\rm(N)}.

\end{enumerata}
{\rm In (a),
the fact that every $(m-1)$-subset of $A$ is nilpotent implies that $A$ is nil
and the fact that $A$ is finitely generated and infinite dimensional implies that $A$ is not nilpotent (Rem.\ \ref{cv9349v20yuj7eeje49rfvh}).  }
\end{nothing}

\begin{nothing*} 
If (P)\,$ \in \{ \text{(N)},  \text{(SN)},  \text{(LN)},  \text{(nil)},  \text{(Lnil)} \}$ and $A$ is an associative algebra satisfying (P),
does it follow that $A_\LL$ satisfies (P)?
Example \ref{PpPPppPPP09vn2n340f30viqX2} shows that there exists an associative algebra $A$ such that $A$ satisfies \text{\rm (SN)} but $A_L$ doesn't.
The following Lemma gives positive results.
\end{nothing*}

\begin{lemma}  \label {0civn23o4fbu34i0r}
Let $A$ be an associative algebra.
\begin{enumerata}

\item \it Let {\rm (P)}$ \in \{ \text{\rm (N)},  \text{\rm (nil)},  \text{\rm (LN)}, \text{\rm (Lnil)} \}$.
If $A$ satisfies {\rm (P)}, then $A_L$ satisfies~{\rm (P)}. 

\item If $A$ satisfies {\rm (SN)}, then $A_L$ satisfies~{\rm (LN)}. 

\end{enumerata}
\end{lemma}

\begin{proof}
(a) (i)
Define $A_n = \Span_\bk\setspec{ a_n \cdots a_1 }{ (a_1, \dots, a_n) \in A^n }$.
If $A$ satisfies (N) then $A_n=0$ for $n$ large enough;
since $[a_n, \dots, a_1] \in A_n$ for all $(a_1, \dots, a_n) \in A^n$, it follows that $A_L$ satisfies (N).

(ii) Suppose that $A$ satisfies (Lnil). Then Prop.\ \ref{GHIEco9vn359r0b3ru6eww}(d) implies that each element of $A$ is nilpotent.
We claim that $A_L$ satisfies (nil).
Indeed, let $x \in A_L$; to prove the assertion, we have to show that $\ad(x) : A_L \to A_L$ is nilpotent.
There exists $m$ such that $x^m=0$ in $A$.
For any $y \in A_L$,  
$$
\big( \ad(x)^{2m-1} \big)(y) = [\underbrace{x, \dots, x}_{2m-1}, y] \in \Span_\bk\setspec{ x^iyx^j }{ i+j=2m-1} .
$$
As $x^i y x^j=0$ for all $i,j$ satisfying $i+j = 2m-1$, we have $\big( \ad(x)^{2m-1} \big)(y) = 0$ for all $y \in A_L$.
So $\ad(x)^{2m-1} = 0$ and consequently $A_L$ satisfies (nil).

(iii) Suppose that $A$ satisfies (LN).
Consider a finite subset $H$ of $A_L$; to show that $A_L$ is (LN),
it suffices to show that $\tilde H$ is nilpotent, where $\tilde H$ denotes the subalgebra of $A_L$ generated by $H$.
Note that $H \subseteq A$ (since $A=A_L$ as sets) and consider the subalgebra $\bar H$ of $A$ generated by $H$.
Then $\bar{H}$ satisfies (N), so part (i) of the proof shows that $(\bar H)_L$ satisfies (N). 
Since $\tilde H \subseteq \bar H$, $\tilde H$ is a subalgebra of the nilpotent Lie algebra $(\bar H)_L$; so $\tilde H$ is nilpotent,
which shows that $A_L$ satisfies (LN).  This proves (a).

(b) Since $A$ satisfies (SN), it satisfies (LN) by Prop.\ \ref{GHIEco9vn359r0b3ru6eww},
so $A_L$ satisfies (LN) by part (a).
\qed\end{proof}

\begin{Example}  \label {PpPPppPPP09vn2n340f30viqX2}
Let $V$ be a vector space of dimension $|\Nat|$ over a field $\bk$.
Then there exists an associative subalgebra $A$ of $\Leul_\bk(V)$ satisfying:
\begin{enumerata}

\item $A$ is (SN) but the Lie algebra $A_L$ is not (SN).

\item $A$ is a uniformly locally nilpotent subset of $\Leul_\bk(V)$.

\end{enumerata}
\end{Example}

\begin{proof}
Let $(e_i)_{i \in \Nat}$ be a basis of $V$,
let $I = \setspec{(i,j) \in \Nat^2}{ i \ge j }$ and, for each $(i,j) \in I$, define the linear map $T_{i,j} : V \to V$ by 
$$
T_{i,j}(e_k) = \begin{cases}
0, & \text{if $k \le i$}, \\
e_j, & \text{if $k > i$.}
\end{cases}
$$
Let $A = \Span_\bk\setspec{ T_{i,j} }{ (i,j) \in I } \subseteq \Leul_\bk(V)$.
Note that if $(i_1,j_1), (i_2,j_2) \in I$ then 
\begin{enumerata}

\item[$(\alpha)$] if  $j_1 \le i_2$ then $T_{i_2,j_2} \circ T_{i_1,j_1} = 0$;
\item[$(\beta)$] if  $j_1 > i_2$ then $(i_1,j_2) \in I$ and $T_{i_2,j_2} \circ T_{i_1,j_1} = T_{i_1,j_2} \in A$.

\end{enumerata}
So $A$ is a subalgebra of $\Leul_\bk(V)$.

Proof that $A$ is (SN).  Consider the subset $\Delta = \setspec{ T_{i,j} }{ (i,j) \in I }$ of $A$.
By Cor.\ \ref{p9v8e674jedrbvysauogUbu}(a), it suffices to show that
for every sequence $(S_0, S_1, \dots) \in \Delta^\Nat$ there exists $n$ such that $S_n S_{n-1} \cdots S_0 = 0$.
In other words, it suffices to show that 
for every sequence $(T_{u_0,v_0}, T_{u_1,v_1}, T_{u_2,v_2}, \dots) \in \Delta^\Nat$ (where $(u_\nu,v_\nu) \in I$ for all $\nu \in \Nat$)
there exists $n$ such that $T_{u_n,v_n} T_{u_{n-1},v_{n-1}} \cdots T_{u_0,v_0} = 0$.
Note that the above statement $(\alpha)$ implies that 
if $(i_1,j_1), (i_2,j_2) \in I$ are such that  $T_{i_2,j_2} \circ T_{i_1,j_1} \neq 0$, then $i_2<i_1$ and $j_2<j_1$.
Consequently, if $T_{u_n,v_n} T_{u_{n-1},v_{n-1}} \cdots T_{u_0,v_0} \neq 0$ then in particular $u_n < u_{n-1} < \cdots < u_0$.
So there must exist an $n$ such that $T_{u_n,v_n} T_{u_{n-1},v_{n-1}} \cdots T_{u_0,v_0} = 0$.

Proof that $A_L$ is not (SN). 
Consider the sequence $(S_{0}, S_{1}, S_{2}, \dots ) \in A^\Nat$, where we define $S_i = T_{i,i}$ for all $i$.
Note that $S_i S_j = 0$ whenever $i \ge j$.
This implies that $[S_1,S_0] = -S_0 S_1$,  $[S_2,S_1,S_0] = [S_2, -S_0 S_1] = S_0S_1S_2$, and by induction we find
$[S_n, \dots, S_0] = (-1)^{n} S_0 \cdots S_n$ for all $n\ge0$.
As $( S_0 \cdots S_n )(e_{n+1}) = e_0$, we have 
$S_0 S_{1} \cdots S_n \neq 0$ and hence $[S_n, \dots, S_0] \neq 0$ (for all $n\ge0$).
So $A_L$ is not (SN) and (a) is proved. 

(b) Let $k \in \Nat$; let us show that $e_k \in \UNil(\Delta)$.
Indeed, if $T_{i,j}$ is any element of $\Delta$ such that $T_{i,j}(e_k) \neq 0$, then $i<k$.
It follows that if 
$(T_{u_1,v_1}, \dots, T_{u_n,v_n}) \in \Delta^n$ satisfies 
$(T_{u_n,v_n} \circ \cdots \circ T_{u_1,v_1})(e_k) \neq 0$ then $k>u_1> \cdots > u_n\ge0$
($T_{u_n,v_n} \circ \cdots \circ T_{u_1,v_1} \neq 0$ implies $u_1> \cdots > u_n$ by the preceding paragraph),
so $n \le k$, proving that $e_k \in \UNil(\Delta)$.
It follows that $\UNil(\Delta) = V$.
As $\bar\Delta = A$, Prop.\ \ref{d0c1i2eb9ew03j} implies that $\UNil(\Delta) = \UNil(A)$, so $\UNil(A) = V$, so (b) is proved.
\qed\end{proof}

\section{Locally nilpotent sets of derivations}
\label {SEC:Localnilpotenceanduniformlocalnilpotenceforderivations}

In this section we assume that $B$ is an algebra over a field $\bk$ (in the sense of \ref{0bg93mm5k6gnndrtuiecj8}(a))
and we consider the Lie subalgebra $\Der_\bk(B)$ of $\Leul_\bk(B)_\LL$ (see \ref{p0vQ398vf83ms43rt7uxcy9w47F}).

\smallskip

A derivation $D \in \Der_\bk(B)$ is said to be  \textit{locally nilpotent} if it is locally nilpotent as a linear map $B \to B$.
We write $\lnd(B)$ for the set of locally nilpotent derivations of $B$, i.e., $\lnd(B) = \locnil(\Leul_\bk B) \cap \Der_\bk(B)$.
For any subset $\Delta$ of $\Der_\bk(B)$, the sets $\UNil(\Delta) \subseteq \Nil(\Delta)$ are subalgebras  of $B$ (by Thm \ref{cp09vv3b49fooOi1wqk498fb}).
If $\Nil(\Delta) = B$,
we say that $\Delta$ is a locally nilpotent subset of $\Der_\bk(B)$.
If $\UNil(\Delta) = B$, we say that $\Delta$ is a uniformly locally nilpotent subset of  $\Der_\bk(B)$.

This section explores Questions 1 and 2, which are stated in the Introduction.
In particular, we give a series of examples showing that both questions have negative answers
when $B$ is the commutative polynomial ring in $| \Nat |$ variables over an arbitrary field $\bk$.
The next section studies the case where $B$ is assumed to satisfy some finiteness condition.

\begin{lemma} \label {pcv0b9i349r09f}
Let $V$ be a vector space over a field $\bk$ and $B$ a commutative polynomial ring over $\bk$ satisfying $\trdeg_\bk(B) = \dim_\bk(V)$.
Then there exist an injective $\bk$-linear map $\nu : V \to B$
and an injective homomorphism of Lie algebras $\psi : \Leul_\bk(V)_\LL \to \Der_\bk(B)$ such that $\bk[ \image\nu] = B$ and 
\begin{equation}  \label {0c2i3urv98703u4t8934jgs9}
\text{for every $F \in \Leul_\bk(V)$, the diagram} \qquad
\raisebox{7mm}{\xymatrix{
B  \ar[r]^{ \psi(F) }  &  B  \\
V \ar[u]^\nu \ar[r]_F  &   V \ar[u]_\nu
}}
\qquad \text{commutes.}
\end{equation}
Moreover, the following statements are true for every subset $\Delta$ of $\Leul_\bk(V)$:
\begin{enumerata}

\item $\Delta \subseteq \locnil(\Leul_\bk V)$  if and only if $\psi(\Delta) \subseteq \lnd(B)$.

\item $\Delta$  is a locally nilpotent subset of $\Leul_\bk(V)$
if and only if $\psi(\Delta)$ is a locally nilpotent subset of $\Der_\bk(B)$.

\item $\Delta$ is a uniformly locally nilpotent subset of $\Leul_\bk(V)$
if and only if $\psi(\Delta)$ is a uniformly locally nilpotent subset of $\Der_\bk(B)$.

\end{enumerata}
\end{lemma}

\begin{proof}
Let $(x_i)_{i \in I}$ be a family of indeterminates over $\bk$ such that $B = \bk[ (x_i)_{i \in I} ]$. 
Then there exists a basis $(e_i)_{i \in I}$ of $V$ indexed by the same set $I$. 
Consider the $\bk$-linear map $\nu : V \to B$ given by $\nu(e_i) = x_i$ for all $i \in I$.
Then $\nu$ is injective and $\nu(V)=B_1$, where we define $B_1 = \Span_\bk\setspec{ x_i }{ i \in I }$.
We might as well assume that $V=B_1$ and that $\nu$ is the inclusion map $B_1 \hookrightarrow B$.
For each $F \in \Leul_\bk(B_1)$, there exists a unique $D_F \in \Der_\bk(B)$ satisfying $D_F(v) = F(v)$ for all $v \in B_1$.
Consider the map $\psi : \Leul_\bk( B_1 ) \to \Der_\bk(B)$, $F \mapsto D_F$.
If $a,b \in \bk$ and $F,G \in \Leul_\bk( B_1 )$ then
the derivations $D_{aF + bG}$ and $aD_F + bD_G$ have the same restriction to $B_1$,
and hence must be equal; so $\psi$ is a $\bk$-linear map (and is clearly injective).
Similarly, the derivations $D_{F \circ G-G \circ F}$ and $D_F \circ D_G - D_G \circ D_F$ have the same restriction to $B_1$,
and hence must be equal.
Thus $\psi$ is a homomorphism of Lie algebras $\Leul_\bk( B_1 )_\LL \to \Der_\bk(B)$ and \eqref{0c2i3urv98703u4t8934jgs9} is true.
To prove (b) and (c), consider a subset $\Delta$ of $\Leul_\bk(V)$.
Since $\Nil( \Delta )$ and $\UNil( \Delta )$ are linear subspaces of $V$ and $(x_i)_{i \in I}$ spans $V$, we have:
$$
\Nil( \Delta ) = V \Leftrightarrow  \, \forall_i\ x_i \in \Nil(\Delta)
\quad \text{and} \quad
\UNil( \Delta ) = V \Leftrightarrow  \, \forall_i\ x_i \in  \UNil(\Delta) .
$$
Since $\Nil( \psi(\Delta) )$ and $\UNil( \psi(\Delta) )$ are
subalgebras of $B = \bk[ (x_i)_{i \in I} ]$ by Thm \ref{cp09vv3b49fooOi1wqk498fb},
and contain $\bk$  by Rem.\ \ref{0f9345eo78u7884r}, we have:
\begin{align*}
\Nil( \psi(\Delta) ) & = B \Leftrightarrow  \, \forall_i\ x_i \in  \Nil(\psi(\Delta)) \\
\UNil( \psi(\Delta) ) & = B \Leftrightarrow  \, \forall_i\ x_i \in  \UNil(\psi(\Delta)) .
\end{align*}
It is easily verified that for each $i \in I$ we have
$$
x_i \in  \Nil(\Delta) \Leftrightarrow x_i \in  \Nil( \psi(\Delta) )
\quad \text{and} \quad
x_i \in  \UNil(\Delta) \Leftrightarrow x_i \in  \UNil( \psi(\Delta) ) .
$$
Consequently, we have 
$\Nil( \Delta ) = V \Leftrightarrow  \Nil( \psi(\Delta) ) = B$ and $\UNil( \Delta ) = V \Leftrightarrow  \UNil( \psi(\Delta) ) = B$,
i.e., assertions (b) and (c) are true. Assertion (a) follows from (b).
Indeed, the condition 
$\Delta \subseteq \locnil(\Leul_\bk V)$ is equivalent to ``for each $F \in \Delta$, $\{F\}$ is a locally nilpotent subset of $\Leul_\bk(V)$;''
by (b), this is equivalent to 
``for each $F \in \Delta$, $\{\psi(F)\}$ is a locally nilpotent subset of $\Der_\bk(B)$,''
which is itself equivalent to $\psi(\Delta) \subseteq \lnd(B)$.
\qed\end{proof}

In view of Question 2, it is natural to ask whether an arbitrary Lie algebra $L$ can be embedded as a 
Lie subalgebra of $\Der_\bk(B)$ (for some $B$) in such a way that $L \subseteq \lnd(B)$.
A preliminary question is whether $L$ can be embedded in $\Der_\bk(B)$ at all (for some $B$).
The answer is affirmative:

\begin{proposition}
Let $L$ be a Lie algebra over a field $\bk$.
Then there exists a commutative polynomial ring $B$ over $\bk$ such that $L$ is isomorphic to a Lie-subalgebra of $\Der_\bk(B)$.
Moreover, if $L$ is finite dimensional then we can choose $B$ to be of finite transcendence degree over $\bk$.
\end{proposition}

\begin{proof}
It is known\footnote{See for instance p.\ 6 of \cite{JacobsonLieAlg}.}
that there exists a vector space $V$ over $\bk$ such that $L$ is isomorphic to a subalgebra of $\Leul_\bk(V)_\LL$.
It is also known that if $\dim L < \infty$ then $V$ can be chosen to be finite dimensional (if $\Char\bk=0$ this is called {\it Ado's Theorem} \cite{Ado};
the general case is due to Iwasawa \cite{Iwasawa1948}).
Consider a polynomial ring $B$ over $\bk$ such that $\trdeg_\bk(B) = \dim_\bk(V)$.
By Lemma \ref{pcv0b9i349r09f}, $\Leul_\bk( V )_\LL$ is isomorphic to a Lie subalgebra of $\Der_\bk(B)$.
As $L$ is isomorphic to a Lie subalgebra of $\Leul_\bk( V )_\LL$, we are done.
\qed\end{proof}


The following is a preliminary step to Ex.\ \ref{2ndPart-Zf24obo87vWLkhkjDTfo499476df}.

\begin{Example}  \label {Zf24obo87vWLkhkjDTfo499476df}
If $\bk$ is a field and $V$ is a $\bk$-vector space of dimension $| \Nat |$,
then there exists an associative subalgebra $A$ of $\Leul_\bk(V)$ satisfying:
\begin{enumerata}

\item $A$ is the free associative algebra on a countably infinite set;

\item $A$ is a uniformly locally nilpotent subset of $\Leul_\bk(V)$.

\end{enumerata}
\end{Example}

\begin{proof}
Let $\bk$ be a field, $\Veul = \{ x_1, x_2, x_3, \dots \}$ a countably infinite set of indeterminates
and $\bk\langle \Veul \rangle$ the polynomial ring over $\bk$ in the \textbf{noncommutative} variables $x_1, x_2, x_3, \dots$.
By a {\it nonempty monomial}, we mean a nonempty finite product of elements of $\Veul$,
i.e., a product  $x_{i_1} \cdots x_{i_n}$ with $n\ge1$.
Let $W \subset \bk\langle \Veul \rangle$ be the set of nonempty monomials and
consider the associative algebra $\Ascr = \Span_\bk W$.
Note that $\Ascr$ is the associative subalgebra $\bar \Veul$ of $\bk\langle \Veul \rangle$ generated by $\Veul$;
so $\Ascr$ is the free associative algebra on $\Veul$.

We say that a nonempty monomial $x_{i_1} \cdots x_{i_n} \in W$ is \textit{admissible} if it satisfies $n>i_n$.
Let $W_0$ be the set of all admissible monomials and consider the $\bk$-subspace $\Ascr_0 = \Span_\bk(W_0)$ of $\Ascr$.
Then $V = \Ascr/\Ascr_0$ is a vector space over $\bk$ of dimension $| \Nat |$.
In fact $\Ascr_0$ is a left ideal of $\Ascr$, so $V$ is also a left $\Ascr$-module;
so if $a \in \Ascr$ and $x \in V$ then $ax \in V$ is defined, and if also $b \in \Ascr$ then $(ab)x = a(bx)$.
If $a \in \Ascr$, let $\mu(a) : V \to V$ be the $\bk$-linear map $x \mapsto ax$ (for $x \in V$).
Then $\mu : \Ascr \to \Leul_\bk( V )$ is a homomorphism of associative $\bk$-algebras.

We claim that $\mu$ is injective.
To see this, 
consider $a \in \Ascr \setminus \{0\}$. 
Write $a = \lambda_1 w_1 + \cdots + \lambda_n w_n$ where $w_1,\dots,w_n$ are distinct elements of $W$, $n\ge1$ and $\lambda_1, \dots, \lambda_n \in \bk^*$.
We have $w_1 = x_{i_1} \cdots x_{i_m}$; then the monomial $w_1 x_{m+1} \in W$ is not admissible.
Since  $a x_{m+1} = \lambda_1 w_1 x_{m+1} + \cdots + \lambda_n w_n x_{m+1}$ where $w_1x_{m+1}$, \dots, $w_nx_{m+1}$ are distinct elements of $W$
and $w_1 x_{m+1} \notin W_0$, $a x_{m+1} \notin \Ascr_0$, so $\big(\mu(a)\big)(x) = ax \neq 0$ where $x = x_{m+1} + \Ascr_0 \in \Ascr/\Ascr_0$,
showing that $\mu(a) \neq 0$.  So $\mu$ is injective.
Define $A = \mu(\Ascr)$, then $A$ is an associative subalgebra of $\Leul_\bk(V)$ and satisfies (a).

We claim that $A$ is a uniformly locally nilpotent subset of $\Leul_\bk(V)$ (meaning $\UNil(A) = V$).
Since $\setspec{ w + \Ascr_0 }{ w \in W }$ is a spanning set for the vector space $V$, 
in order to prove the claim it suffices to show that $w + \Ascr_0 \in \UNil( A )$ for each $w \in W$.
So consider $w \in W$ and write $w = x_{i_1} \cdots x_{i_m}$.
Then for every $(w_1,w_2, \dots, w_{i_m}) \in W^{i_m}$ we have $w_{i_m} \cdots w_2 \cdot w_1 \cdot w \in W_0$
(because $w_1, \dots, w_{i_m}$ are \textit{nonempty} monomials);
it follows that  $a_{i_m} \cdots a_1 \cdot w \in \Ascr_0$ for all $(a_1,a_2, \dots, a_{i_m}) \in \Ascr^{i_m}$,
which implies that  $(F_{i_m} \circ \cdots\circ  F_1)(w+\Ascr_0) = 0$ for all $(F_1, \dots, F_{i_m}) \in A^{i_m}$.
Thus $w + \Ascr_0 \in \UNil(A)$ and consequently  $\UNil(A) = V$.
\qed\end{proof}


The next example is interesting because (a) says that $L$ is as non-nilpotent as a Lie algebra can be, while (b) says 
that it is as locally nilpotent as a subset of $\Der_\bk(B)$ can be. 
So this answers Question 2 in the negative.

\begin{Example}   \label {2ndPart-Zf24obo87vWLkhkjDTfo499476df} 
If $\bk$ is a field and $B$ is the commutative polynomial algebra in $| \Nat |$ variables over $\bk$, 
then there exists a Lie subalgebra $L$ of $\Der_\bk(B)$ satisfying:
\begin{enumerata}

\item $L$ is the free Lie algebra on a countably infinite set;

\item $L$ is a uniformly locally nilpotent subset of $\Der_\bk(B)$.

\end{enumerata}
\end{Example}

\begin{proof}
Write $B = \bk[x_1, x_2, x_3, \dots]$ and let $V = \Span_\bk\{ x_1, x_2, x_3, \dots \} \subseteq B$.
By Ex.\ \ref{Zf24obo87vWLkhkjDTfo499476df},  there exists an associative subalgebra $A$ of $\Leul_\bk(V)$ satisfying:\\[1mm]
$\bullet$\ \  $A$ is the free associative algebra on a countably infinite set;\\
$\bullet$\ \  $A$ is a uniformly locally nilpotent subset of $\Leul_\bk(V)$.\\[1mm]
Consider  a countably infinite subset $S$ of $A$ such that $A$ is the free associative algebra on $S$.
Let $\bar S$ (resp.\  $\tilde S$) be the associative subalgebra of $\Leul_\bk(V)$ (resp. the Lie subalgebra of $\Leul_\bk(V)_\LL$) generated by $S$,
then $\tilde S \subseteq \bar S = A$  and $\tilde S$ is the free  Lie algebra on $S$.
Consider the injective homomorphism of Lie algebras $\psi : \Leul_\bk(V)_\LL \to \Der_\bk(B)$ of \ref{pcv0b9i349r09f}
and define $L = \psi( \tilde S )$. 
Then $L$ is the free Lie algebra on a countably infinite set, and is  a Lie subalgebra of $\Der_\bk(B)$.
We have $\UNil(\tilde S) = \UNil(\bar S) = \UNil(A) = V$, so $\tilde S$ is  a uniformly locally nilpotent subset of $\Leul_\bk(V)$;
then part (c) of  \ref{pcv0b9i349r09f} implies that $L$ is a uniformly locally nilpotent subset of $\Der_\bk(B)$.
\qed\end{proof}

The following gives a negative answer to Question 1.

\begin{Example}   \label {298ehr0c9vvb349t9pwqnfvw9e8} 
If $B$ is the commutative polynomial algebra in $| \Nat |$ variables over a field $\bk$
then there exists a Lie subalgebra $L$ of $\Der_\bk(B)$ satisfying:
\begin{enumerata}

\item $\Nil(L) \neq B$ and if $\Char\bk=0$ then $\Nil(L) = \bk$;

\item every finitely generated subalgebra $L_0$ of $L$ satisfies $\UNil( L_0 ) = B$;

\item the Lie algebra $L$ satisfies (LN).
\end{enumerata}
Note that (b) implies that $L \subseteq \lnd( B )$.
\end{Example}

\begin{proof}
We use the notation $B = \bk[x_0, x_1, x_2, \dots]$.
For each $n \in \Nat$,  define $D_n \in \Der_\bk(B)$ by
$$
D_n( x_i ) =
\begin{cases}
x_{i+1}  &  \text{if $i\le n$}, \\
0 &  \text{if $i > n$.}
\end{cases}
$$
Let $L$ be the Lie subalgebra of $\Der_\bk(B)$ generated by $\{ D_0, D_1, D_2, \dots \}$.
Let $n \in \Nat$; then $x_n \notin \Nil(L)$ because the sequence $(D_n, D_{n+1}, \dots) \in L^\Nat$ 
is such that $(D_N \circ \cdots \circ D_{n+1} \circ D_n)(x_n) = x_{N+1} \neq 0$ for all $N\ge n$.
In particular, $\Nil(L) \neq B$.

Assume that $\Char\bk=0$. To prove $\Nil(L)=\bk$, it suffices to show that 
\begin{equation}  \label {Cc0vh2938hc02C0eudpdfj}
\text{for each $f \in B \setminus \bk$ there exists $n \in \Nat$ such that $D_n(f) \in B \setminus \bk$.}
\end{equation}
Let $f \in B \setminus \bk$. There is a unique $n \in \Nat$ such that $f \in \bk[x_n, x_{n+1}, \dots] \setminus \bk[x_{n+1}, x_{n+2}, \dots]$.
Let $R =  \bk[x_{n+1}, x_{n+2}, \dots]$; then $f = P(x_n)$ for some $P(T) \in R[T] \setminus R$. 
Since $\Char R = 0$ and $R \subseteq \ker(D_n)$, we have $D_n(f) = P'(x_n) D_n(x_n) = x_{n+1}P'(x_n) \in B \setminus\bk$.
This proves \eqref{Cc0vh2938hc02C0eudpdfj}, so $\Nil(L)=\bk$ and (a) is proved.

We now drop the assumption on $\Char\bk$ (so $\bk$ is an arbitrary field until the end of the proof).
Let $L_0$ be a finitely generated Lie subalgebra of $L$.
If we choose $n$ large enough and define $\Delta = \{ D_0, D_1, \dots, D_n \}$
then $L_0 \subseteq \tilde \Delta$, where $\tilde \Delta$ is the Lie subalgebra of $\Der_\bk(B)$ generated by $\Delta$.
Let us fix $n$ with this property.
Note that $\UNil(L_0) \supseteq \UNil(\tilde\Delta)$.
So, to prove (b) and (c), it suffices to show that
\begin{enumerata}
\item[(b$'$)] $\UNil(\tilde\Delta) = B$ \qquad (c$'$) $\tilde\Delta$ is a nilpotent Lie algebra.
\end{enumerata}

Given $\delta \in \Delta = \{ D_0, D_1, \dots, D_n \}$, we have $\delta(x_i) \in \{0,x_{i+1}\}$ for all $i \in \Nat$
and $\delta(x_i)=0$ for all $i>n$.
It follows that
$(\delta_{n+2} \circ \cdots \circ \delta_1)(x_j) = 0$ for all $(\delta_1, \delta_2, \dots, \delta_{n+2}) \in \Delta^{n+2}$ and all $j \in \Nat$, 
i.e., $\delta_{n+2} \circ \cdots \circ \delta_1 = 0$ for all $(\delta_1, \delta_2, \dots, \delta_{n+2}) \in \Delta^{n+2}$.
This certainly implies that $\UNil(\Delta)=B$ (hence  $\UNil(\tilde\Delta) = B$ by Prop.\ \ref{d0c1i2eb9ew03j}),
and it also implies (by Cor.\ \ref{p9v8e674jedrbvysauogUbu}) that $\bar\Delta$ satisfies (N),
where $\bar\Delta$ is the associative subalgebra of $\Leul_\bk(B)$ generated by $\Delta$.
By Lemma \ref{0civn23o4fbu34i0r}, we obtain that $(\bar \Delta)_\LL$ satisfies (N), and since $\tilde\Delta$ is a subalgebra of 
$(\bar \Delta)_\LL$ it follows that $\tilde\Delta$ satisfies (N). So (b$'$) and (c$'$) are proved and we are done.
\qed\end{proof}

In view of the above Example, one may ask whether Question 1 always has an affirmative answer when $L$ is finitely generated.
The answer is negative:

\begin{Example}  \label {cp0b3497m5yokejdrjmk6yf}
Let $B$ be the commutative polynomial algebra in $| \Nat |$ variables over a field $\bk$.
Then for each integer $m\ge2$ there exists an $m$-generated Lie subalgebra $L$ of $\Der_\bk(B)$ satisfying:
\begin{enumerata}

\item $L$ is not a locally nilpotent subset of $\Der_\bk(B)$.

\item Every $(m-1)$-generated Lie subalgebra of $L$ is a locally nilpotent subset of $\Der_\bk(B)$ and a nilpotent Lie algebra.

\item $L \subseteq \lnd(B)$ and $L$ is a nil Lie algebra.

\end{enumerata}
\end{Example}

\begin{proof}
Let $m \ge 2$.  By the result of Golod stated in \ref{covib34oepowd0ccj029}\ref{p0c9n3o49fXvkj4tb9},
there exists an $m$-generated associative $\bk$-algebra $A$ such that $A$ is not nilpotent
but every $(m-1)$-subset of $A$ is nilpotent (so $A$ is nil).
Consider the map $\phi : A \to \Leul_\bk(A)$ of \ref{c0b78h6mrhfwnvbrfvWkdtr8230}.
We noted in \ref{covp32094efdvnpwWzZ} that $\phi$ is a homomorphism of associative algebras. 
Since the $\bk$-vector space $A$ has dimension $| \Nat |$, we have $\dim_\bk(A) = \trdeg_\bk(B)$;
so we may consider an injective $\bk$-linear map $\nu : A \to B$
and an injective homomorphism of Lie algebras $\psi : \Leul_\bk(A)_\LL \to \Der_\bk(B)$ satisfying the conditions of Lemma \ref{pcv0b9i349r09f} (with $V=A$);
in particular, $B=\bk[ \image\nu ]$ and the diagram in part (i) of \eqref{983gh5vbbmBvnd2736t4udhi} commutes for every $F \in \Leul_\bk(A)$.
Let $H$ be a generating set of $A$ with $| H | = m$,
let $\Delta = \phi(H) \subset \Leul_\bk(A)$,
let $\Lambda = \psi(\Delta) \subset \Der_\bk(B)$ and let $L$ be the Lie subalgebra of $\Der_\bk(B)$ generated by $\Lambda$.
Let us prove that $L$ satisfies the desired conditions.
\begin{equation}  \label {983gh5vbbmBvnd2736t4udhi}
\mbox{(i)} \quad \xymatrix{
B  \ar[r]^{ \psi(F) }  &  B  \\
A \ar[u]^\nu \ar[r]^F  &   A \ar[u]_\nu
}
\qquad
\mbox{(ii)} \quad \xymatrix@R=20pt{
A  \ar[r]^-{\phi}   &    \Leul_\bk(A) = \Leul_\bk(A)_\LL   \ar[r]^-{\psi} & \Der_\bk(B)    \\
A \ar @{=}[u] \ar[r]^-{\text{surjective}}  &   \bar{\Delta} \supseteq \tilde{\Delta} \ar@{^{(}->}[u] \ar[r]^-{\isom} &  L \ar@{^{(}->}[u] 
}
\end{equation}

Let $\bar\Delta$ (resp.\ $\tilde\Delta$) be the subalgebra of $\Leul_\bk(A)$ (resp.\ of $\Leul_\bk(A)_\LL$) generated by $\Delta$.
Since $A$ is generated by $H$ and $\phi$ is a homomorphism of associative algebras, $\phi(A) = \bar \Delta$.
Since $\psi$ is an injective homomorphism of Lie algebras, it restricts to an isomorphism $\tilde\Delta \to L$.
See part (ii) of \eqref{983gh5vbbmBvnd2736t4udhi}.

Since $A$ is finitely generated but not nilpotent, Prop.\ \ref{GHIEco9vn359r0b3ru6eww}(b) implies that $A$ is not (SN).
So there exists an infinite sequence $(a,a_0,a_1,a_2,\dots) \in A^\Nat$ such that $a_n \cdots a_0 \cdot a \neq 0$ for all $n \in \Nat$.
If we define  $F_i = \phi(a_i)$ for all $i\in\Nat$ then $(F_0,F_1,\dots) \in (\bar\Delta)^\Nat$ satisfies $(F_n \circ \cdots \circ F_0)(a) \neq 0$ for all $n$,
so $a \notin \Nil( \bar \Delta ) = \Nil( \Delta )$ (cf.\ Prop.\ \ref{d0c1i2eb9ew03j}) and hence $\Delta$ is not a locally nilpotent subset of $\Leul_\bk(A)$.
By  Lemma \ref{pcv0b9i349r09f}(a), $\Lambda = \psi(\Delta)$ is not a locally nilpotent subset of $\Der_\bk(B)$. In particular, (a) is true.

To prove (b), consider an $(m-1)$-subset $\Lambda'$ of $L$ and let $L'$ be the subalgebra of $L$ generated by $\Lambda'$.
There exists an $(m-1)$-subset $H'$ of $A$ such that $\psi(\phi(H')) = \Lambda'$.
By our choice of $A$, $H'$ is a nilpotent subset of $A$;
it follows that $\Delta' = \phi(H')$ is a nilpotent subset of $\bar \Delta$, i.e.,
there exists $n>0$ such that
\begin{equation} \label {pc0v34kjsos7e}
\text{$F_n \circ \cdots \circ F_1 = 0$ for all $(F_1, \dots, F_n) \in (\Delta')^n$.}
\end{equation}
In particular, $\Delta'$ is a locally nilpotent subset of $\Leul_\bk(A)$;
by Lemma \ref{pcv0b9i349r09f}(b), $\Lambda' = \psi(\Delta')$ is a locally nilpotent subset of $\Der_\bk(B)$,
so Prop.\ \ref{d0c1i2eb9ew03j} implies that $L'$ is  a locally nilpotent subset of $\Der_\bk(B)$.
On the other hand, \eqref{pc0v34kjsos7e}
and Cor.\ \ref{p9v8e674jedrbvysauogUbu} imply that the associative algebra $\overline{\Delta'}$ is nilpotent.
By Lemma \ref{0civn23o4fbu34i0r}, it follows that the Lie algebra  $(\overline{\Delta'})_\LL$ is nilpotent.
Since $\widetilde{\Delta'}$ is a subalgebra of  $(\overline{\Delta'})_\LL$, $\widetilde{\Delta'}$  is nilpotent.
Since $\psi$ maps  $\widetilde{\Delta'}$ onto $L'$, we obtain that $L'$ is nilpotent.
This proves (b).

We have $L \subseteq \lnd(B)$ by (b).
Since $\bar\Delta$ is a homomorphic image of the nil algebra $A$, $\bar\Delta$ is nil (as an associative algebra).
So the Lie algebra $(\bar\Delta)_\LL$ is nil, by Lemma \ref{0civn23o4fbu34i0r}.
As $\tilde\Delta$ is a subalgebra of $(\bar\Delta)_\LL$, it is a nil Lie algebra.
Since $L \isom \tilde\Delta$, $L$ is nil. So (c) is proved.
\qed\end{proof}

\begin{corollary} \label {p0c9vn349Cf0cpe0a}
Let $B$ be the commutative polynomial algebra in $| \Nat |$ variables over a field $\bk$.
\begin{enumerata}

\item There exists an infinite subset $\Delta$ of $\Der_\bk(B)$ satisfying:
\begin{itemize}

\item $\Delta$ is not a locally nilpotent subset of $\Der_\bk(B)$;

\item every finite subset of $\Delta$ is a locally nilpotent subset of $\Der_\bk(B)$.

\end{itemize}

\item For each integer $m\ge2$, there exists an $m$-subset $\Delta$ of $\Der_\bk(B)$ satisfying:
\begin{itemize}

\item $\Delta$ is not a locally nilpotent subset of $\Der_\bk(B)$;

\item every proper subset of $\Delta$ is a locally nilpotent subset of $\Der_\bk(B)$.

\end{itemize}

\end{enumerata}
\end{corollary}

\begin{proof}
The following is an obvious consequence of Prop.\ \ref{d0c1i2eb9ew03j}:
\begin{enumerate}

\item[$(*)$] \rm For any subset $\Delta$ of $\Der_\bk(B)$,
$\Delta$ is a locally nilpotent subset of $\Der_\bk(B)$ 
if and only if 
the Lie subalgebra of $\Der_\bk(B)$ generated by $\Delta$ is a locally nilpotent subset of $\Der_\bk(B)$.

\end{enumerate}

\noindent Proof of (a). Consider the Lie subalgebra $L$ of $\Der_\bk(B)$ given in  Ex.\ \ref{298ehr0c9vvb349t9pwqnfvw9e8}
and let $\Delta$ be any generating set of $L$.
Then, by $(*)$, $\Delta$ satisfies (a).

\noindent Proof of (b).
Let $m\ge2$. Consider the Lie subalgebra $L$ of $\Der_\bk(B)$ given by Ex.\ \ref{cp0b3497m5yokejdrjmk6yf}.
Let $\Delta$ be a generating set of $L$ with $| \Delta | = m$.
Then, by $(*)$, $\Delta$ satisfies (b).
\qed\end{proof}

\section{The case of derivation-finite algebras}
\label{SEC:thecaseofderivationfinitealgebras}

\begin{Definition} \label {c9fweCdefderfinitedc9v9ueibqovp}
Let $B$ be an algebra over a field $\bk$.
We say that $B$ is \textit{derivation-finite} if there exists a finite subset $X$ of $B$ satisfying:
$$
\text{if $D \in \Der_\bk(B)$ satisfies $D(x)=0$ for all $x \in X$, then $D=0$.}
$$
\end{Definition}

For instance,
if $B$ is a finitely generated $\bk$-algebra then it is derivation-finite.
If $\Char\bk=0$ and $B$ is a commutative integral domain of finite transcendence degree over $\bk$ then $B$ is derivation-finite.

This section re-examines Question 2 under the assumption
that $B$ is a derivation-finite algebra over a field $\bk$.
We begin with an example.

\begin{Example} \label {928349rfer8d9}
Consider the polynomial ring $B = \bk[X_1, \dots, X_n]$ where $n \ge 2$ and $\bk$ is a field of characteristic zero.
Let $L$ be the set of all $D \in \Der_\bk(B)$ satisfying $D(X_1) \in \bk$ and
$D(X_i) \in \bk[X_1, \dots, X_{i-1}]$ for all $i = 2, \dots, n$.
Note that $L$ is a Lie subalgebra of $\Der_\bk(B)$. 
We claim:
\begin{enumerata}
\item $L$ is a locally nilpotent subset of $\Der_\bk(B)$.
\item $L$ is not a uniformly locally nilpotent subset of $\Der_\bk(B)$.
\item The Lie algebra $L$ satisfies {\rm (SN)} but not {\rm (nil)}.
\end{enumerata}
\end{Example}

\begin{proof}
The proof of (a) uses the following trivial observation: {\it if $b \in B$ is such that $D(b) \in \Nil(L)$ for all $D \in L$, then $b \in \Nil(L)$.}
We have $D(X_1) \in \bk \subseteq \Nil(L)$ for all $D \in L$,
so $X_1 \in  \Nil(L)$, so $\bk[X_1] \subseteq \Nil(L)$ by Thm \ref{cp09vv3b49fooOi1wqk498fb}.
Let $i \in \{2, \dots, n\}$ be such that  $\bk[X_1, \dots, X_{i-1}] \subseteq \Nil(L)$.
Since $D(X_i) \in \bk[X_1, \dots, X_{i-1}] \subseteq \Nil(L)$ for all $D \in L$,
$X_i \in \Nil(L)$ and hence  $\bk[X_1, \dots, X_{i}] \subseteq \Nil(L)$ by Thm \ref{cp09vv3b49fooOi1wqk498fb}.
It follows by induction that $\Nil(L) = B$, so (a) is true.

(b) Given any $m>0$ let $D_m = X_1^m \, \frac{\partial}{\partial X_2}$, and let $E = \frac{\partial}{\partial X_1}$;
then $\{D_m,E\} \subset L$ and (since $\Char\bk=0$) $(E^m \circ D_m)(X_2) \neq 0$, so $X_2 \notin \UNil(L)$.

(c) Let us check that $L$ is not a nil Lie algebra.
Let $E = \frac{\partial}{\partial X_1} \in L$. We claim that $\ad(E) : L \to L$ is not nilpotent.
Indeed, let $m>0$ and let us prove that $\ad(E)^m \neq 0$.
It suffices to show that $\ad(E)^m(D_m) = [E, \dots, E, D_m]$ is not zero
(where $D_m = X_1^m \, \frac{\partial}{\partial X_2}$ and where there are $m$ ``$E$'' in the bracket);
so it suffices to show that  $[E, \dots, E, D_m](X_2) \neq 0$.
Since $E(X_2) = 0$, we have $[E, \dots, E, D_m](X_2) =  (E^m \circ D_m)(X_2) \neq 0$.
So $\ad(E)$ is not nilpotent and consequently the Lie algebra $L$ is not nil.

To show that the Lie algebra $L$ is (SN), we use the following notation.
Let $B_0=\bk$ and $B_j = \bk[X_1,\dots,X_j]$ for all $j = 1, \dots, n$.
If $j\in\{1,\dots,n\}$ and $f \in B_{j-1}$, define $D_f^j = f \frac{\partial}{\partial X_j }\in L$.
Let $H$ be the set of $D_f^j$, for all pairs $(j,f)$ such that $j\in\{1,\dots,n\}$ and $f \in B_{j-1}$ (thus $H \subseteq L$).
Consider $D_f^j, D_g^k \in H$ (where $f \in B_{j-1}$ and $g \in B_{k-1}$).
A straightforward calculation gives $[D_f^j, D_g^k] = 0$ when $j=k$, and $[D_f^j, D_g^k] = D^k_{ D_f^j(g) }$ if $j<k$.
Since $D_f^j(g) = 0$ when $j=k$, we get
\begin{equation} \label {cp0vb4jros8erghbe9}
[D_f^j, D_g^k] = D^k_{ D_f^j(g) } \quad \text{whenever $j\le k$.}
\end{equation}

We have $L = \Span_\bk(H)$, so in particular $H$ is a generating set for the Lie algebra $L$.
By Cor.\ \ref{p9v8e674jedrbvysauogUbu}(a), 
to show that $L$ is (SN) it suffices to show that for every sequence $(d_0,d_1,d_2,\dots) \in H^\Nat$
there exists an $m$ such that $[d_m, \dots, d_0]=0$.
So consider $(d_0,d_1,d_2,\dots) \in H^\Nat$.
Then  for each $i$ we have $d_i = D_{f_i}^{j_i}$, where $j_i \in \{1, \dots, n\}$ and $f_i \in \bk[X_1,\dots,X_{j_i-1}]$.
Choose $\nu \in \Nat$ such that $j_{\nu} = \max\setspec{ j_i }{ i \in \Nat }$.
Then \eqref{cp0vb4jros8erghbe9} implies that 
$$
\text{for each $m \ge \nu$, $[d_m, \dots, d_0] = D^{j_\nu}_{g_m}$ for some $g_m \in B_{j_{\nu}-1}$.}
$$
We have
$[d_{m+1},d_m, \dots, d_0] = [d_{m+1}, [d_m, \dots, d_0]] = [d_{m+1}, D^{j_\nu}_{g_m}] = [D^{j_{m+1}}_{f_{m+1}}, D^{j_\nu}_{g_m}] \\
= D^{j_\nu}_{ D^{j_{m+1}}_{f_{m+1}}( g_m ) }$,
so $g_{m+1} =  D^{j_{m+1}}_{f_{m+1}}( g_m ) = d_{m+1}(g_m)$ (for all $m\ge\nu$). Thus $g_m = (d_m \circ \cdots \circ d_{\nu+1})(g_\nu)$ for all $m > \nu$.
Since $(d_i)_{i=\nu+1}^\infty$  is an infinite sequence in $H$, and $H$ is a locally nilpotent subset of $\Der_\bk(B)$,
it follows that for $m$ large enough we have $g_m=0$ and hence $[d_m, \dots, d_0] = 0$.
This shows that the Lie algebra $L$ is (SN).
\qed\end{proof}

Until the end of the section, we assume:
$$
\textbf{\boldmath $B$ is an algebra over a field $\bk$ (cf.\ \ref{0bg93mm5k6gnndrtuiecj8}(a)) and is derivation-finite.}
$$

For each element $D$ of the Lie algebra $\Der_\bk(B)$, we may consider the map $\ad(D) : \Der_\bk(B) \to \Der_\bk(B)$, $E \mapsto [D,E]$ 
(this is the map defined in \ref{pa9enbr2o39efcp230wsd}).
Then we have the following.

\begin{lemma}  \label {0vc94xtgh6uy6aiw39ihfb}
If $D \in \lnd(B)$ then the map $\ad(D) : \Der_\bk(B) \to \Der_\bk(B)$ is locally nilpotent.
\end{lemma}

\begin{proof}
We have to show that for each $E \in \Der_\bk(B)$ there exists $N>0$ such that $\ad(D)^N(E) = 0$.  Let $E \in \Der_\bk(B)$.
Choose a nonempty finite subset $X$ of $B$ such that 
$$
\text{if $F \in \Der_\bk(B)$ satisfies $F(x)=0$ for all $x \in X$, then $F=0$.}
$$
Choose $n>0$ such that $D^n(x) = 0$ for all $x \in X$.
Since 
$$
Y = \setspec{ (E \circ D^i)(x) }{ (i,x) \in \{0,\dots,n-1\} \times X }
$$
is a finite subset of $B$, there exists $m>0$ such that $D^m(y)=0$ for all $y \in Y$;
then $(D^m \circ E \circ D^i)(x) = 0$  for all $(i,x) \in \{0,\dots,n-1\} \times X$. 
Let $F = \ad(D)^{m+n-1}(E) = [D,D,\dots, D,E]$.
Since $F \in \Der_\bk(B)$, in order to show that $F=0$ it suffices to show that $F(x)=0$ for all $x \in X$.
Let $x \in X$; then $F(x)$ is a linear combination (over $\bk$) of terms of the form 
$(D^{j} \circ E \circ D^i)(x)$ where $i,j\ge0$ and $i+j=m+n-1$.
In each one of these terms we have either $i\ge n$ or ($i<n$ and $j\ge m$), so $(D^{j} \circ E \circ D^i)(x)=0$. 
It follows that $F(x)=0$ and hence that $\ad(D)^{m+n-1}(E) = 0$.
\qed\end{proof}

The next facts give some positive answers to Question 2. 
Observe that the conclusion of Cor.\ \ref{cijvnn394npebg67872u}(a) cannot be strenghtened from (Lnil) to (nil), by Ex.\ \ref{928349rfer8d9}.

\begin{corollary} \label {cijvnn394npebg67872u}
Let $L$ be a Lie subalgebra of $\Der_\bk(B)$ such that $L \subseteq \lnd(B)$.
\begin{enumerata}

\item $L$ is locally nil \mbox{\rm(Lnil)}.

\item If $L$ is finite dimensional then it is a nilpotent Lie algebra.

\end{enumerata}
\end{corollary}

\begin{proof}
Assertion (a) follows immediately from Lemma \ref{0vc94xtgh6uy6aiw39ihfb}.
Part (b) follows from the fact (Prop.\ \ref{GHIEco9vn359r0b3ru6eww})
that (Lnil) $\Leftrightarrow$ (N) for finite dimensional Lie algebras.
\qed\end{proof}

\begin{proposition}  \label {p90cb2473dq739e3r627}
For a finitely generated Lie subalgebra $L$ of $\Der_\bk(B)$, the following are equivalent:
\begin{enumerata}

\item $L$ is a nilpotent Lie algebra;
\item $L$ is a Lie-locally nilpotent subset of $\Der_\bk(B)$;
\item $L$ is a uniformly Lie-locally nilpotent subset of $\Der_\bk(B)$.

\end{enumerata}
Moreover, if $L$ is a locally nilpotent subset of $\Der_\bk(B)$ then \text{\rm (a--c)} are satisfied.
\end{proposition} 

\begin{proof}
Choose a finite subset $\Delta$ of $L$ such that $\tilde\Delta = L$.
Choose a nonempty finite subset $X$ of $B$ such that 
$$
\text{if $F \in \Der_\bk(B)$ satisfies $F(x)=0$ for all $x \in X$, then $F=0$.}
$$

It is obvious that (a) implies (b). 
If (b) is true then $\Delta$ is a finite  Lie-locally nilpotent subset of $\Der_\bk(B)$; by  Lemma \ref{Nil-0cccno12q9wdnp}, 
$\Delta$ is uniformly Lie-locally nilpotent; so $L = \tilde\Delta$ is uniformly Lie-locally nilpotent
by  Prop.\ \ref{Nil-d0c1i2eb9ew03j}, and this shows that (b) implies (c).

Assume that (c) is true.
Since $X$ is a finite set and  $L$ is uniformly Lie-locally nilpotent, there exists $N \in \Nat$ satisfying 
\begin{equation} \label {FJcob34598yt34e9k}
\text{ $[D_{n}, \dots, D_1](x)=0$\quad for all $n\ge N$, $(D_1, \dots, D_{n}) \in L^{n}$ and $x \in X$.} 
\end{equation}
In \eqref{FJcob34598yt34e9k}, the fact that $[D_{n}, \dots, D_1]$ is a derivation that annihilates each element of $X$ implies that 
$[D_{n}, \dots, D_1]=0$. So we have 
$[D_{N}, \dots, D_1]=0$ for all $(D_1, \dots, D_{N}) \in L^{N}$, i.e., $L$ is nilpotent.
So (c) implies (a) and consequently the three conditions are equivalent.

Assume that $L$ is a locally nilpotent subset of $\Der_\bk(B)$.
Then $\Delta$ is a finite locally nilpotent subset of $\Der_\bk(B)$, so 
$\Delta$ is uniformly locally nilpotent by  Lemma \ref{0cccno12q9wdnp},
so $L=\tilde\Delta$ is  uniformly locally nilpotent by Prop.\ \ref{d0c1i2eb9ew03j},
so $L$ is  uniformly Lie-locally nilpotent by Lemma \ref{p9vcuwbIhEj483u839}, so (a--c) are satisfied.
\qed\end{proof}

\bibliographystyle{alpha}

\end{document}